\RequirePackage{amsmath}
\documentclass[runningheads]{llncs}

\usepackage[T1]{fontenc}
\usepackage{graphicx}
\usepackage{fourier}
\usepackage{cite}
\usepackage{subcaption}
\usepackage{amsfonts}
\usepackage{comment}
\usepackage{csquotes}

\usepackage{amsthm}

\usepackage{array}
\usepackage{multirow}
\usepackage[colorlinks=true, allcolors=blue]{hyperref}
\usepackage{tikz}
\usetikzlibrary{calc}

\newtheorem{conj}{Conjecture}
\newtheorem{ques}{Question}

\definecolor{bluereadable}{HTML}{1A80BB}
\definecolor{redreadable}{HTML}{F2C45F}

\begin{document}

\title{On a conjecture of Faudree and Schelp}
%
%
\author{Jan Goedgebeur\inst{1,2}
\and
Jorik Jooken\inst{1}
\and
Michiel Provoost\inst{1}
\and
Carol T. Zamfirescu\inst{2,3}
}
\authorrunning{Goedgebeur, Jooken, Provoost, Zamfirescu}

\institute{Department of Computer Science, KU Leuven Campus Kulak-Kortrijk, 8500 Kortrijk, Belgium\\
\email{\{jan.goedgebeur,jorik.jooken,michiel.provoost\}@kuleuven.be}\\ \and
Department of Mathematics, Computer Science and Statistics, Ghent University, 9000 Ghent, Belgium\\ \and
Department of Mathematics, Babe\c{s}-Bolyai University, Cluj-Napoca, Roumania\\
\email{czamfirescu@gmail.com}}
\maketitle
\begin{abstract}
In 1976 Faudree and Schelp conjectured that in a hamiltonian-con-nected graph on $n$ vertices, any two distinct vertices are connected by a path of length $k$ for every $k \ge n/2$. In 1978 Thomassen constructed a (non-cubic and non-planar) family of counterexamples, showing that there exist hamiltonian-connected $n$-vertex graphs containing two vertices with no path of length $n-2$ between them. We complement this result by describing cubic planar counterexamples on $6p+16$ vertices, each containing vertices between which there is no path of any odd length greater than $1$ and at most $4p+9$. Motivated by a remark of Thomassen about a gap in the cycle spectrum of hamiltonian-connected graphs, we also describe an infinite family of hamiltonian-connected graphs with many gaps in the first half of their cycle spectra.


\smallskip

\textbf{MSC 2020}: 05C38, 05C45, 05C10, 05C85

\smallskip

\textbf{Keywords}: Hamiltonian-connected, Path length distribution, planar graph, cubic graph, counterexample
\end{abstract}
\section{Introduction}
In a given graph $G$, a path between vertices $u$ and $v$ in $G$ will be called a \textit{$uv$-path}. A path or cycle in $G$ that visits every vertex of $G$ is \textit{hamiltonian}, and if $G$ contains such a path (cycle), $G$ is \textit{traceable} (\textit{hamiltonian}). $G$ is \textit{hamiltonian-connected} if for every pair of distinct vertices $u$ and $v$ of $G$ there is a hamiltonian $uv$-path. For a path $P$, its \textit{length} is $|E(P)|$. In this paper we are interested in the following conjecture.
\begin{conj}[Faudree and Schelp~\cite{FS}]
For every hamiltonian-connected graph $G$ on $n$ vertices and every pair of distinct vertices $u, v$ in $G$, there exists a path between $u$ and $v$ of length $k$ for every $k$ with $\frac{n}{2} \leq k \leq n-1$.
\end{conj}

Unless explicitly stated otherwise, the order of a given graph will be denoted by~$n$. We say that a graph $G$ \textit{has property} ${\mathfrak P}_k$ if for any two distinct vertices $u$ and $v$ in $G$ there is a $uv$-path in $G$ with $k$ vertices, i.e.\ of length $k - 1$. Thus, a graph is hamiltonian-connected if and only if it satisfies ${\mathfrak P}_n$. In this notation, the conjecture of Faudree and Schelp states that ${\mathfrak P}_n$ implies $\bigwedge_{\frac{n}{2}+1 \le k \le n} {\mathfrak P}_k$.


In their paper, Faudree and Schelp prove the conjecture when restricted to graphs that are the the square of a block, and  several other results were published which support the conjecture~\cite{FRS73,FS74,FS,F76}. However, in the general case, Thomassen~\cite{thomassen1978counterexamples} described an infinite family of counterexamples. (We remark that these graphs described by Thomassen also provide counterexamples to~\cite[Conjectures 3.1 and 3.4]{almohanna2018hamiltonian}.) Whenever a conjecture is disproven, it is interesting to investigate whether the conjecture can be adapted to obtain a correct statement either by restricting the class of graphs or by weakening the property that was claimed. For example, for the conjecture of Faudree and Schelp one could try to restrict the class of graphs to be planar or regular or adapt the interval in which the path lengths lie. In the present paper we present progress in both directions.

The rest of this article is structured as follows. We end the introduction by mentioning related work. In Section~\ref{sec:planar_counterexamples} we show that the smallest counterexample to the Faudree-Schelp conjecture is surprisingly small: there exists a cubic planar $8$-vertex hamiltonian-connected graph having two distinct vertices with no path of length~$4$ between them (whereas Thomassen's smallest counterexample has $27$ vertices). Moreover, we show that the Faudree-Schelp conjecture is false in a strong sense by constructing an infinite family of cubic planar graphs, such that for every integer $p \geq 1$ there exists a graph in our family on $16+6p$ vertices containing two vertices such that there is no path between them whose length is in $\{3,5,7,\ldots,9+4p\}$. We point out that Thomassen's counterexamples are very different in nature from ours, as his examples do not satisfy ${\mathfrak P}_{n-1}$. The examples are thus at two `ends' of the spectrum. What happens in-between is largely unknown, i.e.\ it is possible that a more conservative version of the Faudree-Schelp Conjecture, with a smaller set of ${\mathfrak P}_k$'s that must be satisfied, does hold. 

Thomassen states that, as a consequence of the Faudree-Schelp Conjecture, every hamiltonian-connected graph would be the union of its $(n-1)$-cycles. (Note that more generally, for $\ell \ge 3$, $G$ satisfying ${\mathfrak P}_\ell$ implies that every edge of $G$ is contained in a cycle of length $\ell$, so $G$ is the union of its $\ell$-cycles.) In~\cite{thomassen1978counterexamples} Thomassen describes a hamiltonian-connected graph with $n$ vertices and an edge that is not contained in any cycle of length $n - 1$. He goes on to write that maybe there even exists a hamiltonian-connected graph containing no cycle of length $n - 1$ at all. Motivated by this remark, in Section~\ref{sec:porous_cycle_spectrum} we present an infinite family of hamiltonian-connected graphs whose cycle spectrum---the \textit{cycle spectrum} of a graph $G$ is the set of integers $\ell$ such that $G$ contains an $\ell$-cycle---has size 2/3 of the graphs' order, noting that all gaps occur in the first half of the cycle spectrum. Gaps in the second half of the spectrum elude us. Finally, in Section~\ref{sec:conclusion} we discuss additional questions that, we believe, deserve attention.

We refer the reader to~\cite{survey} for more related work and a survey on path length distributions in graphs that satisfy $\mathfrak{P}_k$ for certain values of $k$. We close this section by 
mentioning two old and likely difficult problems in the area which we find intriguing. Kotzig~\cite{K79} conjectured in the seventies that for every integer $k \ge 3$, there is no graph such that between every pair of vertices there is exactly one path of length $k$. For more context on this interesting conjecture, see~\cite{W16}. Another appealing structural problem is due to Thomassen, who mentions in~\cite{thomassen1978counterexamples} that we do not know whether $\bigwedge_{k \ge 3, k \ne n - 1} \mathfrak{P}_k$ implies $\mathfrak{P}_{n-1}$.

\section{Cubic planar counterexamples to the Faudree-Schelp Conjecture}
\label{sec:planar_counterexamples}

Straightforward computations yield that there are three pairwise non-isomorphic counterexamples to the Faudree-Schelp Conjecture on eight vertices, see Fig.~\ref{fig:smallest}. It is easy to verify that these graphs are hamiltonian-connected and that they contain vertices (in Fig.~\ref{fig:smallest} labelled $3$ and $4$) between which there is no path of length $4$. We checked with a computer that no smaller counterexamples exist\footnote{The code to replicate this verification can be found on Github~\cite{github}.}. We refer the interested reader to the Appendix for more details on the algorithms that were used as well as additional computational results that led to new counterexamples in various classes of graphs. We summarise:

\begin{proposition}
The graphs shown in Fig.~\ref{fig:smallest} have property ${\mathfrak P}_{n}$ but not property ${\mathfrak P}_{n/2+1}$. These are the smallest counterexamples to the Faudree-Schelp Conjecture.
\end{proposition}


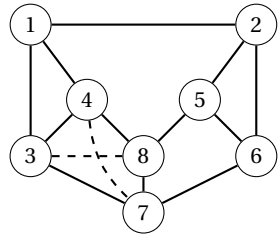
\begin{figure}
\vspace{-5mm}
    \centering
    \begin{tikzpicture}[main_node/.style={circle,draw,minimum size=1em,inner sep=3pt},scale=0.5]

\node[main_node] (0) at (-6, 0) {1};
\node[main_node] (1) at (0, 0) {2};
\node[main_node] (2) at (-6, -3.5) {3};
\node[main_node] (3) at (-4.5, -2) {4};
\node[main_node] (4) at (-1.5, -2) {5};
\node[main_node] (5) at (0, -3.5) {6};
\node[main_node] (6) at (-3, -5) {7};
\node[main_node] (7) at (-3, -3.5) {8};

 \path[draw, thick]
(0) edge node {} (1) 
(0) edge node {} (2) 
(0) edge node {} (3) 
(1) edge node {} (4) 
(1) edge node {} (5) 
(2) edge node {} (3) 
(2) edge node {} (6) 
(3) edge node {} (7) 
(4) edge node {} (5) 
(4) edge node {} (7) 
(5) edge node {} (6) 
(6) edge node {} (7) 
;

 \path[draw,thick, dashed]
 (2) edge node {} (7)
  (3) edge [bend right=20] node {} (6)
 ;
\end{tikzpicture}
    \caption{An $8$-vertex graph that does not satisfy ${\mathfrak P}_{5}$. It is the smallest counterexample to the Faudree-Schelp Conjecture, both in terms of order and size, and happens to be cubic and planar. It retains the property of being a counterexample even if at least one of the dashed edges is added, but if both are added, planarity is lost. No other non-isomorphic counterexamples of order 8 exist.}
    \label{fig:smallest}
    \vspace{-5mm}
\end{figure}

Henceforth, in a graph containing vertices $u$ and $v$, we denote by $P_{uv}$ a $uv$-path which, notation-wise, we view as an ordered sequence of vertices $u, \ldots, v$. Moreover, when we make use of such a path $P_{uv}$ without further specifying its internal vertices, we always mean that there exists a suitable $uv$-path for the given context. By $\overleftarrow{P_{uv}}$ we denote the aforementioned path $P_{uv}$ in reverse order, i.e.\ a particular $vu$-path.

\subsection{Cubic planar counterexamples with pairs of vertices missing many path lengths}

Let $H_0$ be the cubic planar $16$-vertex-graph shown in Fig.~\ref{fig:smallest-k}. For each integer $k \geq 1$, we define the graph $H_k$ on $16+6k$ vertices as the graph obtained by taking $H_{k-1}$ and replacing a graph induced by seven vertices (isomorphic with the graph shown on the left of Fig.~\ref{fig:expansionMoreComplicated}) by the graph shown on the right of Fig.~\ref{fig:expansionMoreComplicated}. Note that $H_{k-1}[\{a,b,c,d,e,f,g\}]$ is isomorphic with $H_k[\{a',a'',b',b'',c,c',c''\}]$, so one can always find seven such vertices. Clearly, $H_k$ is planar and cubic. In the next three auxiliary results we first summarise some useful properties of $H_0$, leaving the straightforward proof to the reader, and then show that $H_k$ is hamiltonian-connected and that $H_k$ contains two vertices such that there are many $k$ for which there is no path of length $k$ between them.

\begin{lemma}
\label{lem:baseCaseH}
    $H_0$ is hamiltonian-connected and there is no path of length $3$, $5$, $7$ or $9$ between vertices $a$ and $c$.
\end{lemma}

\begin{figure}
\vspace{-5mm}
    \centering
    \begin{tikzpicture}[main_node/.style={circle,draw,minimum size=1em,inner sep=1pt}]

\node[main_node] (0) at (-2, 2) {};
\node[main_node] (1) at (-3, 1) {};
\node[main_node] (2) at (-1, 2) {};
\node[main_node] (3) at (0, 1) {};
\node[main_node] (4) at (-5, 3) {};
\node[main_node] (5) at (-4, 3) {};
\node[main_node] (6) at (-3, 3) {};
\node[main_node] (7) at (-4.5, 4) {};
\node[main_node] (8) at (-3.5, 4.5) {};
\node[main_node] (9) at (-0.5, 3) {$f$};
\node[main_node] (10) at (0.25, 3) {$a$};
\node[main_node] (11) at (0.75, 2.5) {$b$};
\node[main_node] (12) at (1.5, 2.5) {$e$};
\node[main_node] (13) at (0.75, 3.5) {$c$};
\node[main_node] (14) at (1.5, 3.5) {$d$};
\node[main_node] (15) at (0, 4.5) {$g$};

\path[draw, thick]
(7) edge node {} (5)
(5) edge node {} (4)
(4) edge node {} (7)
(7) edge node {} (8)
(8) edge node {} (6)
(6) edge node {} (5)
(15) edge node {} (14)
(14) edge node {} (13)
(14) edge node {} (12)
(11) edge node {} (12)
(11) edge node {} (13)
(13) edge node {} (10)
(10) edge node {} (11)
(15) edge node {} (9)
(9) edge node {} (10)
(9) edge node {} (2)
(3) edge node {} (12)
(2) edge node {} (3)
(0) edge node {} (2)
(1) edge node {} (3)
(6) edge node {} (0)
(1) edge node {} (4)
(0) edge node {} (1)
(8) edge node {} (15)
;

\end{tikzpicture}
    \caption{The graph $H_0$ is a counterexample to the Faudree-Schelp Conjecture.}
    \label{fig:smallest-k}
\end{figure}
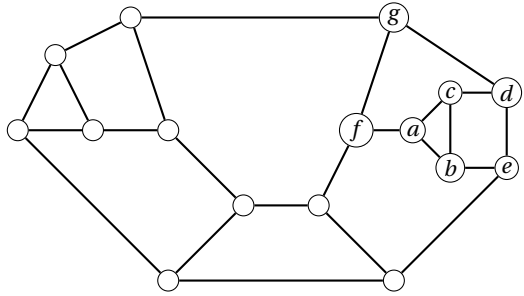

\begin{figure}
    \centering
    \begin{tikzpicture}[main_node/.style={circle,draw,minimum size=1em,inner sep=1pt},scale=0.7]

\node[main_node,draw=none] (0) at (-4, -3) {};
\node[main_node] (1) at (-3, -1) {$a$};
\node[main_node,draw=none] (2) at (0, 2) {};
\node[main_node] (3) at (-2, 2) {$g$};
\node[main_node] (4) at (0, 0) {$d$};
\node[main_node] (5) at (-1.5, 0) {$c$};
\node[main_node] (6) at (-1.5, -2) {$b$};
\node[main_node] (7) at (0, -2) {$e$};
\node[main_node] (8) at (-4, -1) {$f$};
\node[main_node,draw=none] (9) at (0, -3) {};

\path[draw, thick]
    (5) edge node {} (1)
    (1) edge node {} (6)
    (6) edge node {} (5)
    (0) edge node {} (8)
    (8) edge node {} (1)
    (3) edge node {} (8)
    (3) edge node {} (4)
    (4) edge node {} (5)
    (6) edge node {} (7)
    (4) edge node {} (7)
    (2) edge node {} (3)
    (7) edge node {} (9);

\node[main_node,draw=none] (20) at (6, -3) {};
\node[main_node] (21) at (7, -2) {$a$};
\node[main_node,draw=none] (22) at (10, 3) {};
\node[main_node] (23) at (8, 3) {$g$};
\node[main_node] (24) at (10, 2) {$d$};
\node[main_node] (25) at (9, 2) {$c$};
\node[main_node] (26) at (10, -2) {$b$};
\node[main_node] (27) at (12, -2) {$e$};
\node[main_node] (28) at (6, -2) {$f$};
\node[main_node,draw=none] (29) at (12, -3) {};
\node[main_node] (210) at (7.5, -0.25) {$a'$};
\node[main_node] (211) at (9, 0.5) {$c''$};
\node[main_node] (212) at (9, -1) {$b''$};
\node[main_node] (213) at (10, 0.5) {$c'$};
\node[main_node] (214) at (10, -1) {$b'$};
\node[main_node] (215) at (8.5, -0.25) {$a''$};

\path[draw, thick]
(21) edge node {} (26)
(20) edge node {} (28)
(28) edge node {} (21)
(23) edge node {} (28)
(23) edge node {} (24)
(24) edge node {} (25)
(26) edge node {} (27)
(24) edge node {} (27)
(22) edge node {} (23)
(27) edge node {} (29)
(210) edge node {} (21)
(210) edge node {} (25)
(25) edge node {} (213)
(213) edge node {} (211)
(211) edge node {} (212)
(212) edge node {} (214)
(214) edge node {} (213)
(211) edge node {} (215)
(215) edge node {} (212)
(210) edge node {} (215)
(214) edge node {} (26);

    \draw [->,very thick] (1.5,0) -- (4.5,0);

\end{tikzpicture}
    \caption{$H_k$ is obtained by starting from $H_{k-1}$ and performing the operation shown here.}
    \label{fig:expansionMoreComplicated}
\end{figure}
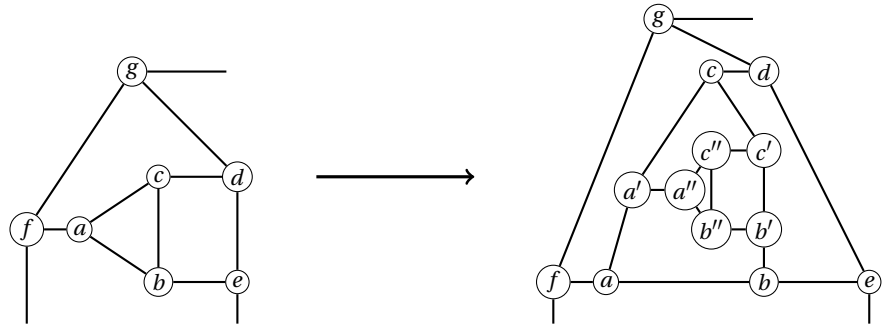

\begin{lemma}
\label{lem:HamConnH}
    For each integer $k \geq 0$, $H_k$ is hamiltonian-connected.
\end{lemma}
\begin{proof}
We prove this by induction on $k$. The base case $k=0$ is given by Lemma~\ref{lem:baseCaseH}. Assume the statement holds for all $H_i$ with $i<k$. Let $H_k$ be constructed from $H_{k-1}$ as shown in Fig.~\ref{fig:expansionMoreComplicated}. The proof strategy is to show the existence of a hamiltonian path between each two distinct vertices $u$ and $v$ in $H_k$ by properly modifying certain long paths between two distinct vertices in $H_{k-1}$.

Throughout the proof, we see $H_{k-1} - ac - bc$ as subgraph of $H_k$. Let $u, v \in V(H_{k-1})$ be distinct vertices and let $P$ be a hamiltonian $uv$-path in $H_{k-1}$. For a subpath $S$ of $P$, we denote by $\rho(S)$ the path replacing $S$ in order to obtain a hamiltonian $uv$-path in $H_k$. We have $1 \le |E(P) \cap \{ab,ac,bc\}| \le 2$. We obtain a hamiltonian $uv$-path in $H_k$ by considering the replacements \begin{multline*}
\rho(ab) = aa'a''b''c''c'b'b, \ \rho(ac) = aa'a''c''b''b'c'c, \ \rho(bc) = bb'c'c''b''a''a'c,\\
\rho(abc) = abb'c'c''b''a''a'c, \ \rho(bca) = bb'b''a''c''c'ca'a, \ \rho(cab) = cc'b'b''c''a''a'ab.
\end{multline*}


Next, we show that there is a hamiltonian $uv$-path in $H_k$ for all $u \in \{a',a'',b',b'',c',$ $c''\}$ and $v \in V(H_k)\setminus\{a,a',a'',b,b',b'',c,c',c''\}$. Let $P$ be a hamiltonian $bv$-path in $H_{k-1}$. Note that $ac \in E(P)$. There are three cases to consider. 

\smallskip

\noindent \textsc{Case 1}: $bac$ is a subpath of $P$. We replace $bac$ by one of the following paths in order to obtain suitable hamiltonian $uv$-paths in $H_k$ (one for each aforementioned vertex $u$):
\vspace{-1mm}
\begin{align*}
a'abb'b''a''c''c'c, \quad b'baa'a''b''c''c'c, \quad c'c''a''b''b'baa'c,\\
a''b''c''c'b'baa'c, \quad b''b'baa'a''c''c'c , \quad c''b''a''a'abb'c'c.
\end{align*}

\smallskip

\noindent \textsc{Case 2}: $bca$ is a subpath of $P$. In this situation, $v \ne f$. If $v=d$, then $P$ is of the form $bcP_{ae}d$. Then the following paths are hamiltonian $uv$-paths in $H_k$: 
\vspace{-1mm}
\begin{align*}
a'cc'c''a''b''b'bP_{ae}d, \quad b'b\overleftarrow{P_{ae}}a'a''b''c''c'cd, \quad c'ca'a''c''b''b'bP_{ae}d,\\
a''a'cc'c''b''b'bP_{ae}d, \quad b''c''a''a'cc'b'bP_{ae}d, \quad c''c'ca'a''b''b'bP_{ae}d.
\end{align*}
If $v \neq d$, then $gde$ or $edg$ is a subpath of $P$. Now there are two subcases for $P$. In the first subcase, $P$ is of the form $bcaP_{fe}dP_{gv}$, then $bacdP_{ef}P_{gv}$ is also a hamiltonian $bv$-path in $H_{k-1}$. It contains $bac$ as a subpath and this was precisely the previous case. In the second subcase, $P$ is of the form $bcaP_{fg}dP_{ev}$. Then, putting $P_{bv}=bP_{ev}$, the following paths are hamiltonian $uv$-paths in $H_k$: 
\begin{align*}
a'aP_{fg}dcc'c''a''b''b'P_{bv}, \ \ a''a'aP_{fg}dcc'c''b''b'P_{bv}, \ \ b'c'c''b''a''a'cdP_{fg}aP_{bv},\\
b''b'c'c''a''a'cd\overleftarrow{P_{fg}}aP_{bv}, \ \ c'b'b''c''a''a'cd\overleftarrow{P_{fg}}aP_{bv}, \ \ c''c'b'b''a''a'cd\overleftarrow{P_{fg}}aP_{bv}.
\end{align*}


\smallskip

\noindent \textsc{Case 3}: $be$ is a subpath of $P$. Note that $P$ cannot have $edg$ as subpath as this would imply $v = c$, which is impossible by the choice of $v$. There are four subcases for $P$ depending on whether $f$ or $g$ occurs first when following $P$ from $b$ to $v$. In the first, second, and third subcase, $P$ is of the form $bP_{ef}acdP_{gv}$ or $bP_{eg}dcaP_{fv}$ or $bP_{eg}facd$. In these subcases, the paths $bacdP_{ef}P_{gv}$ and $bacdP_{eg}P_{fv}$ and $bcaf\overleftarrow{P_{eg}}d$ are also hamiltonian $bv$-paths in $H_{k-1}$, which correspond exactly with Case 1, Case 1 and Case 2, respectively. In the fourth subcase, $P$ is of the form $bedcafP_{gv}$. Then, putting $P_{av}=afP_{gv}$, the following paths are hamiltonian $uv$-paths in $H_k$: 
\begin{align*}
a'a''c''b''b'c'cdebP_{av}, \quad a''b''c''c'b'bedca'P_{av}, \quad b'c'c''b''a''a'cdebP_{av},\\
b''a''c''c'b'bedca'P_{av}, \quad c'cdebb'b''c''a''a'P_{av}, \quad c''c'cdebb'b''a''a'P_{av}.
\end{align*}

Finally, we show that there is a hamiltonian $uv$-path in $H_k$ for all $u \in \{a',a'',b',b'',c',$ $c''\}$ and $v \in \{a,a',a'',b,b',b'',c,c',c''\}$ with $u \neq v$.

Let $P$ be a hamiltonian $fd$-path in $H_{k-1}$. $P$ is of the form $fP_{ge}bacd$ or $facbP_{eg}d$. Thus there exists a path in $H_{k-1}$ with endpoints $g$ and $e$ containing all vertices in $V(G)\setminus \{a,b,c,d,f\}$. We may assume without loss of generality that this path is $P_{ge}$. Then the following paths are the desired hamiltonian paths in $H_k$: 
\vspace{-1mm}
\begin{align*}
a'cd\overleftarrow{P_{ge}}fabb'c'c''b''a'', \quad a'afP_{ge}dcc'c''a''b''b'b, \quad a'a''b''c''c'cd\overleftarrow{P_{ge}}fabb',\\
a'cd\overleftarrow{P_{ge}}fabb'c'c''a''b'', \quad a'a''b''c''c'b'bafP_{ge}dc, \quad a'cd\overleftarrow{P_{ge}}fabb'b''a''c''c',\\
a'a''b''b'bafP_{ge}dcc'c'', \quad a''a'afP_{ge}dcc'c''b''b'b, \quad a''a'cd\overleftarrow{P_{ge}}fabb'c'c''b'',\\
a''a'cd\overleftarrow{P_{ge}}fabb'b''c''c', \quad bafP_{ge}dca'a''b''c''c'b', \quad bb'c'cd\overleftarrow{P_{ge}}faa'a''c''b'',\\
bafP_{ge}dca'a''c''b''b'c', \quad bb'c'cd\overleftarrow{P_{ge}}faa'a''b''c'', \quad b'bafP_{ge}dca'a''b''c''c'\\
b''c''a''a'cd\overleftarrow{P_{ge}}fabb'c', \quad b''a''a'cd\overleftarrow{P_{ge}}fabb'c'c'', \quad c'b'bafP_{ge}dca'a''b''c''.
\end{align*}


Let $P$ be a hamiltonian $ga$-path in $H_{k-1}$. $P$ is of the form $gP_{fe}dcba$ or $gdcbP_{ef}a$. Thus there exists a path in $H_{k-1}$ with endpoints $e$ and $f$ containing all vertices in $V(G)\setminus \{a,b,c,d,g\}$. We may assume without loss of generality that this path is $P_{fe}$. Now the following paths are hamiltonian paths in $H_k$:
\vspace{-1mm}
\begin{gather*}
\begin{align*}
ab\overleftarrow{P_{fe}}gdcc'b'b''c''a''a', \quad aa'cdgP_{fe}bb'c'c''b''a'', \quad ab\overleftarrow{P_{fe}}gdca'a''b''c''c'b',\\
aa'cdgP_{fe}bb'c'c''a''b'', \quad aa'cdgP_{fe}bb'b''a''c''c', \quad ab\overleftarrow{P_{fe}}gdca'a''b''b'c'c'',\\
a''a'ab\overleftarrow{P_{fe}}gdcc'c''b''b', \quad a''a'ab\overleftarrow{P_{fe}}gdcc'b'b''c'', \quad b'c'cdgP_{fe}baa'a''c''b'',
\end{align*}\\[0.5mm]
\ \quad b'c'cdgP_{fe}baa'a''b''c''.
\end{gather*}


Let $P$ be a hamiltonian $ec$-path in $H_{k-1}$. $P$ is of the form $ebaP_{fg}dc$ or $edP_{gf}abc$. Thus there exists a path in $H_{k-1}$ with endpoints $g$ and $f$ containing all vertices in $V(G)\setminus \{a,b,c,d,e\}$. We may assume without loss of generality that this path is $P_{gf}$. Now the following paths are hamiltonian paths in $H_k$:
\vspace{-1mm}
\begin{gather*}
\begin{align*}
a''b''c''c'b'bedP_{gf}aa'c, \quad b'bedP_{gf}aa'a''b''c''c'c, \quad b''a''c''c'b'bedP_{gf}aa'c,
\end{align*}\\[0.5mm]
ca'a\overleftarrow{P_{gf}}debb'b''a''c''c', \quad cc'b'bedP_{gf}aa'a''b''c''.
\end{gather*}


We have shown that there exists for each pair of distinct $u, v \in V(H_k)$ a hamiltonian $uv$-path in $H_k$. This completes the proof.
\end{proof}

We now show that $H_k$ contains two vertices such that the lengths of the paths between them are very restricted:
\begin{lemma}
\label{lem:LengthH}
    For each integer $k \geq 0$, $H_k$ contains two adjcent vertices such that there is no path between them whose length is in $\{3,5,7,\ldots,9+4k\}$.
\end{lemma}
\begin{proof}
We prove this by induction on $k$. The base case $k=0$ is given by Lemma~\ref{lem:baseCaseH}. Now assume the statement holds for all $H_i$ with $i<k$. Let $H_k$ be constructed from $H_{k-1}$ as shown in Fig.~\ref{fig:expansionMoreComplicated}. As before, throughout the proof we see $H_{k-1} - ac - bc$ as subgraph of $H_k$. We will now prove that there is no $a''c''$-path in $H_k$ whose length is in $\{3,5,7,\ldots,9+4k\}$. By enumerating all short $a''c''$-paths, one can conclude that there is clearly no such path of length $3$, $5$ or $7$. For the sake of obtaining a contradiction, let $P$ be an $a''c''$-path whose length is in $\{9,11,\ldots,9+4k\}$. Now $P$ must contain at least one of the vertices from $\{a,b,c\}$. For every way in which (some of) these three vertices could appear in $P$, we shall obtain a contradiction:

\smallskip

\noindent \textsc{Case 1}: $P \in \{ a''a'P_{ab}b'b''c'', a''a'P_{ab}b'c'c'' \}$. Now $|E(P_{ab})|=|E(P)|-5$ and $P_{ab}$ is a path containing only vertices in $V(H_{k-1})\setminus \{c\}$. Therefore, $P_{ab}c$ is a path of length $|E(P)|-4$ in $H_{k-1}$, a contradiction.

\smallskip

\noindent \textsc{Case 2}: $P \in \{ a''a'P_{ac}c'c'', a''a'P_{ac}c'b'b''c'' \}$. Now $|E(P_{ac})|=|E(P)|-4$ or $|E(P_{ac})|=|E(P)|-6$ and $P_{ac}$ is a path containing only vertices in $V(H_{k-1})$, a contradiction.

\smallskip

\noindent \textsc{Case 3}: $P = a''b''b'P_{ba}a'cc'c''$. Now $|E(P_{ba})|=|E(P)|-7$ and $P_{ba}$ is a path containing only vertices in $V(H_{k-1})\setminus \{c\}$. Therefore $P'=cP_{ba}$ is a path of length $|E(P)|-6$ in $H_{k-1}$, a contradiction.

\smallskip

\noindent \textsc{Case 4}: $P = a''b''b'P_{bc}c'c''$. This case has three subcases:\\
In the first subcase, $P = a''b''b'bP_{ac}c'c''$. Now $|E(P_{ac})|=|E(P)|-6$ and $P_{ac}$ is a path in $H_{k-1}$, a contradiction. In the second subcase, $P = a''b''b'bP_{ea}a'cc'c''$. Now $|E(P_{ea})|=|E(P)|-8$ and $P_{ea}$ is a path containing only vertices in $V(H_{k-1})\setminus \{b,c\}$. Therefore, $P'=cbP_{ea}$ is a path of length $|E(P)|-6$ in $H_{k-1}$, a contradiction. In the third subcase, $P = a''b''b'bP_{ed}cc'c''$. Now $|E(P_{ed})|=|E(P)|-7$ and $P_{ed}$ is a path containing only vertices in $V(H_{k-1})\setminus \{a,b,c\}$. Therefore, $abP_{ed}c$ is a path of length $|E(P)|-4$ in $H_{k-1}$, a contradiction.

\smallskip

\noindent \textsc{Case 5}: $P \in \{ a''a'P_{cb}b'b''c'', a''a'P_{cb}b'c'c'' \}$. This case has two subcases:\\
In the first subcase, $P = a''a'P_{ce}bb'b''c''$ or $a''a'P_{ce}bb'c'c''$. Now $|E(P_{ce})|=|E(P)|-6$ and $P_{ce}$ is a path containing only vertices in $V(H_{k-1})\setminus \{a,b\}$. Therefore, $P'=P_{ce}ba$ is a path of length $|E(P)|-4$ in $H_{k-1}$, a contradiction. In the second subcase, $P \in \{ a''a'P_{ca}bb'b''c'', a''a'P_{ca}bb'c'c'' \}$. Now $|E(P_{ca})|=|E(P)|-6$ and $P_{ca}$ is a path in $H_{k-1}$, a contradiction. \end{proof}

By combining Lemmas~\ref{lem:HamConnH} and~\ref{lem:LengthH}, we obtain that the Faudree-Schelp Conjecture is false in a strong sense:
\begin{theorem}
\label{thm:HFamily}
For every non-negative integer $k$ there exists a hamiltonian-connected graph on $16+6k$ vertices that does not satisfy $\mathfrak{P}_{i}$ for each $i \in \{4,6,8,\ldots,10+4k\}$, i.e.\ a linear number of times (in the order of the graph) within the interval conjectured by Faudree and Schelp.
\end{theorem}

\section{Hamiltonian-connected graphs missing many cycle lengths}
\label{sec:porous_cycle_spectrum}

In the previous section, we showed the existence of hamiltonian-connected graphs having two adjacent vertices with no path of length $\ell$ between them, for many values of $\ell$. In other words, such graphs have \textit{some} edge that is not contained in any cycle of length $\ell+1$. In the current section we discuss hamiltonian-connected graphs such that \textit{no} edge is contained in a cycle of length $\ell+1$, i.e.\ hamiltonian-connected graphs without $(\ell+1)$-cycles. This is motivated by Thomassen's remark from~\cite{thomassen1978counterexamples} stating that
\textquote{[m]aybe there even exists a hamiltonian-connected graph containing no cycle of length $n-1$ at all.} In the following, based on three auxiliary results, we shall present an infinite family of hamiltonian-connected graphs with small cycle spectra (relative to their orders). It shall sometimes be more convenient to use the following notation. For integers $a, b$ with $a \le b$, let $[a,b] := \{ k \in \mathbb{Z} : a \le k \le b \}$, and for a natural number $n$, let $[n] := \{ 1, \ldots, n \}$. 

Let $G_A$ be the graph shown in Fig.~\ref{fig:gadget}. For each integer $k \geq 2$, let $F_k$ be the graph obtained by taking $k$ copies of $G_A$ (labelling the resulting vertices as $v_{i,j}$, $i \in [k], j \in [6]$, in the natural way) and adding the edges $v_{i,3}v_{i+1,1}$ and $v_{i,6}v_{i+1,4}$ for all $i \in [k]$. In the remainder of this paper, whenever we speak of $v_{i,j}$ in the graph $F_k$, $i$ should be interpreted as $((i-1)~{\rm mod}~k) + 1$ and $j$ should be interpreted as $((j-1)~{\rm mod}~6) + 1$.
\vspace{-5mm}
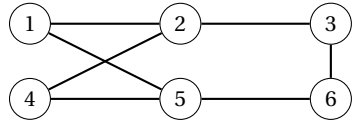
\begin{figure}[h]
        \centering
        \begin{tikzpicture}[main_node/.style={circle,draw,minimum size=1em,inner sep=3pt]}]

\node[main_node] (0) at (-2, 0) {4};
\node[main_node] (1) at (-2, 1) {1};
\node[main_node] (2) at (0, 0) {5};
\node[main_node] (3) at (0, 1) {2};
\node[main_node] (4) at (2,0) {6};
\node[main_node] (5) at (2, 1) {3};

 \path[draw, thick]
(0) edge node {} (2) 
(2) edge node {} (4) 
(4) edge node {} (5) 
(5) edge node {} (3) 
(3) edge node {} (1) 
(0) edge node {} (3) 
(1) edge node {} (2) 
;

\end{tikzpicture}
        \caption{The graph \(G_A\) with all vertices labelled.}
        \label{fig:gadget}
        \vspace{-5mm}
    \end{figure}

\begin{lemma}
\label{lem:FHamConn}
For each integer $k \geq 2$, $F_k$ is hamiltonian-connected.
\end{lemma}

\begin{proof}
We will show the existence of a hamiltonian $uw$-path in $F_k$ for every two distinct vertices $u,w \in V(F_k)$. This will be done by constructing several subpaths in each of the $k$ copies of the graph $G_A$ and then gluing these subpaths together to obtain a hamiltonian $uw$-path.

Suppose $u=v_{i,p}$ and $w=v_{i,q}$ for some $i \in [k]$ and $1 \leq p < q \leq 6$. Up to symmetry there are nine distinct pairs $(p,q)$ to consider. Fig.~\ref{fig:intragadget} shows for each such pair two subpaths, where one subpath has vertex $u$ as an endpoint and the other subpath has vertex $w$ as an endpoint. These two subpaths contain all vertices in $\{v_{i,1},v_{i,2},v_{i,3},v_{i,4},$ $v_{i,5},v_{i,6}\}$. Fig.~\ref{fig:intragadget-rest} contains three additional sets of subpaths that contain all vertices in $\{v_{j,1},v_{j,2},$ $v_{j,3},v_{j,4},v_{j,5},v_{j,6}\}$. A hamiltonian $uw$-path can now be constructed by gluing a set of subpaths shown in Fig.~\ref{fig:intragadget} (in copy $i$ of the graph $G_A$ in $F_k$) together with a set of subpaths shown in Fig.~\ref{fig:intragadget-rest} (in copy $j \neq i$ of the graph $G_A$ in $F_k$). More precisely, the subpaths depicted in the cases shown in Fig.~\ref{fig:case2-1-1},~\ref{fig:case2-1-2},~\ref{fig:case2-1-4},~\ref{fig:case2-1-5},~\ref{fig:case2-1-6},~\ref{fig:case2-1-8} and~\ref{fig:case2-1-9} (in copy $i$ of the graph $G_A$) should be combined with the subpaths shown in Fig.~\ref{fig:case2-2-1} (in copy $j \neq i$ of the graph $G_A$). The subpaths depicted in the case shown in Fig.~\ref{fig:case2-1-3} (in copy $i$ of the graph $G_A$) should be combined with the subpaths shown in Fig.~\ref{fig:case2-2-3} (in copy $i-1$ of the graph $G_A$) and the subpaths shown in Fig.~\ref{fig:case2-2-2} (in copy $j \notin \{i, i-1\}$ of the graph $G_A$). Finally, the subpaths depicted in the case shown in Fig.~\ref{fig:case2-1-7} (in copy $i$ of the graph $G_A$) should be combined with the subpaths shown in Fig.~\ref{fig:case2-2-2} (in copy $j \neq i$ of the graph $G_A$).

\begin{figure}
\centering
     \begin{subfigure}[b]{0.3\linewidth}
    \begin{tikzpicture}[main_node/.style={circle,draw,minimum size=1em,inner sep=2pt},scale=0.8]

\node[main_node] (4) at (-1.5, 0) {};
\node[main_node,fill=bluereadable] (1) at (-1.5, 1) {};
\node[main_node] (5) at (0, 0) {};
\node[main_node,fill=bluereadable] (2) at (0, 1) {};
\node[main_node] (6) at (1.5,0) {};
\node[main_node] (3) at (1.5, 1) {};

 \path[draw, thick]
(4) edge node {} (5) 
(5) edge node {} (6) 
(6) edge node {} (3) 
(3) edge node {} (2) 
(2) edge node {} (1) 
(4) edge node {} (2) 
(1) edge node {} (5) 
;

 \path[draw, line width=0.8mm, color=bluereadable]
(2) edge node {} (4) 
(4) edge node {} (5) 
(5) edge node {} (6)
(6) edge node {} (3) 
;

\draw[gray, thick] (4) -- (-2.5,0);
\draw[bluereadable, line width=0.8mm] (1) -- (-2.5,1);
\draw[gray, thick] (6) -- (2.5,0);
\draw[bluereadable, line width=0.8mm] (3) -- (2.5,1);

\end{tikzpicture}
\vspace{-5.5mm}
    \caption{Case 1}
    \label{fig:case2-1-1}
    \end{subfigure}
    \hfill
    \centering
         \begin{subfigure}[b]{0.3\linewidth}
    \begin{tikzpicture}[main_node/.style={circle,draw,minimum size=1em,inner sep=2pt},scale=0.8]

\node[main_node] (4) at (-1.5, 0) {};
\node[main_node,fill=bluereadable] (1) at (-1.5, 1) {};
\node[main_node] (5) at (0, 0) {};
\node[main_node] (2) at (0, 1) {};
\node[main_node] (6) at (1.5,0) {};
\node[main_node,fill=bluereadable] (3) at (1.5, 1) {};

 \path[draw, thick]
(4) edge node {} (5) 
(5) edge node {} (6) 
(6) edge node {} (3) 
(3) edge node {} (2) 
(2) edge node {} (1) 
(4) edge node {} (2) 
(1) edge node {} (5) 
;
 \path[draw, line width=0.8mm, color=bluereadable]
(1) edge node {} (5) 
(2) edge node {} (4) 
(3) edge node {} (2)
(5) edge node {} (6) 
;

\draw[bluereadable, line width=0.8mm] (4) -- (-2.5,0);
\draw[gray, thick] (1) -- (-2.5,1);
\draw[bluereadable, line width=0.8mm] (6) -- (2.5,0);
\draw[gray, thick] (3) -- (2.5,1);

\end{tikzpicture}
\vspace{-5.5mm}
    \caption{Case 2}
    \label{fig:case2-1-2}
    \end{subfigure}
    \hfill
    \begin{subfigure}[b]{0.3\linewidth}
    \begin{tikzpicture}[main_node/.style={circle,draw,minimum size=1em,inner sep=2pt},scale=0.8]

\node[main_node,fill=bluereadable] (4) at (-1.5, 0) {};
\node[main_node,fill=bluereadable] (1) at (-1.5, 1) {};
\node[main_node] (5) at (0, 0) {};
\node[main_node] (2) at (0, 1) {};
\node[main_node] (6) at (1.5,0) {};
\node[main_node] (3) at (1.5, 1) {};

 \path[draw, thick]
(4) edge node {} (5) 
(5) edge node {} (6) 
(6) edge node {} (3) 
(3) edge node {} (2) 
(2) edge node {} (1) 
(4) edge node {} (2) 
(1) edge node {} (5) 
;

 \path[draw, line width=0.8mm, color=bluereadable]
(1) edge node {} (2) 
(2) edge node {} (3) 
(4) edge node {} (5)
(5) edge node {} (6) 
;

\draw[gray, thick] (4) -- (-2.5,0);
\draw[gray, thick] (1) -- (-2.5,1);
\draw[bluereadable, line width=0.8mm] (6) -- (2.5,0);
\draw[bluereadable, line width=0.8mm] (3) -- (2.5,1);
\end{tikzpicture}
\vspace{-5.5mm}
    \caption{Case 3}
    \label{fig:case2-1-3}
    \end{subfigure}
    \hfill
    \centering
\begin{subfigure}[b]{0.3\linewidth}
\vspace{3mm}
    \begin{tikzpicture}[main_node/.style={circle,draw,minimum size=1em,inner sep=2pt},scale=0.8]

\node[main_node] (4) at (-1.5, 0) {};
\node[main_node,fill=bluereadable] (1) at (-1.5, 1) {};
\node[main_node,fill=bluereadable] (5) at (0, 0) {};
\node[main_node] (2) at (0, 1) {};
\node[main_node] (6) at (1.5,0) {};
\node[main_node] (3) at (1.5, 1) {};

 \path[draw, thick]
(4) edge node {} (5) 
(5) edge node {} (6) 
(6) edge node {} (3) 
(3) edge node {} (2) 
(2) edge node {} (1) 
(4) edge node {} (2) 
(1) edge node {} (5) 
;

 \path[draw, line width=0.8mm, color=bluereadable]
(5) edge node {} (4) 
(2) edge node {} (1) 
(3) edge node {} (6)
(2) edge node {} (3) 
;

\draw[bluereadable, line width=0.8mm] (4) -- (-2.5,0);
\draw[gray, thick] (1) -- (-2.5,1);
\draw[bluereadable, line width=0.8mm] (6) -- (2.5,0);
\draw[gray, thick] (3) -- (2.5,1);

\end{tikzpicture}
\vspace{-5.5mm}
    \caption{Case 4}
    \label{fig:case2-1-4}
    \end{subfigure}
    \hfill
\begin{subfigure}[b]{0.3\linewidth}
    \begin{tikzpicture}[main_node/.style={circle,draw,minimum size=1em,inner sep=2pt},scale=0.8]

\node[main_node] (4) at (-1.5, 0) {};
\node[main_node,fill=bluereadable] (1) at (-1.5, 1) {};
\node[main_node] (5) at (0, 0) {};
\node[main_node] (2) at (0, 1) {};
\node[main_node,fill=bluereadable] (6) at (1.5,0) {};
\node[main_node] (3) at (1.5, 1) {};

 \path[draw, thick]
(4) edge node {} (5) 
(5) edge node {} (6) 
(6) edge node {} (3) 
(3) edge node {} (2) 
(2) edge node {} (1) 
(4) edge node {} (2) 
(1) edge node {} (5) 
;

 \path[draw, line width=0.8mm, color=bluereadable]
(5) edge node {} (6) 
(4) edge node {} (5) 
(2) edge node {} (4)
(2) edge node {} (3) 
;

\draw[gray, thick] (4) -- (-2.5,0);
\draw[bluereadable, line width=0.8mm] (1) -- (-2.5,1);
\draw[gray, thick] (6) -- (2.5,0);
\draw[bluereadable, line width=0.8mm] (3) -- (2.5,1);

\end{tikzpicture}
    \vspace{-5.5mm}
    \caption{Case 5}
    \label{fig:case2-1-5}
    \end{subfigure}
    \hfill
         \begin{subfigure}[b]{0.3\linewidth}
    \begin{tikzpicture}[main_node/.style={circle,draw,minimum size=1em,inner sep=2pt},scale=0.8]

\node[main_node] (4) at (-1.5, 0) {};
\node[main_node] (1) at (-1.5, 1) {};
\node[main_node] (5) at (0, 0) {};
\node[main_node,fill=bluereadable] (2) at (0, 1) {};
\node[main_node] (6) at (1.5,0) {};
\node[main_node,fill=bluereadable] (3) at (1.5, 1) {};

 \path[draw, thick]
(4) edge node {} (5) 
(5) edge node {} (6) 
(6) edge node {} (3) 
(3) edge node {} (2) 
(2) edge node {} (1) 
(4) edge node {} (2) 
(1) edge node {} (5) 
;

 \path[draw, line width=0.8mm, color=bluereadable]
(3) edge node {} (6) 
(2) edge node {} (1) 
(1) edge node {} (5)
(5) edge node {} (4) 
;

\draw[bluereadable, line width=0.8mm] (4) -- (-2.5,0);
\draw[gray, thick] (1) -- (-2.5,1);
\draw[bluereadable, line width=0.8mm] (6) -- (2.5,0);
\draw[gray, thick] (3) -- (2.5,1);

\end{tikzpicture}
\vspace{-5.5mm}
    \caption{Case 6}
    \label{fig:case2-1-6}
    \end{subfigure}
    
    \centering
    \hfill
     \begin{subfigure}[b]{0.3\linewidth}
     \vspace{3mm}
    \begin{tikzpicture}[main_node/.style={circle,draw,minimum size=1em,inner sep=2pt},scale=0.8]

\node[main_node] (4) at (-1.5, 0) {};
\node[main_node] (1) at (-1.5, 1) {};
\node[main_node,fill=bluereadable] (5) at (0, 0) {};
\node[main_node,fill=bluereadable] (2) at (0, 1) {};
\node[main_node] (6) at (1.5,0) {};
\node[main_node] (3) at (1.5, 1) {};

 \path[draw, thick]
(4) edge node {} (5) 
(5) edge node {} (6) 
(6) edge node {} (3) 
(3) edge node {} (2) 
(2) edge node {} (1) 
(4) edge node {} (2) 
(1) edge node {} (5)
;
 \path[draw, line width=0.8mm, color=bluereadable]
(4) edge node {} (5) 
(6) edge node {} (3) 
(2) edge node {} (1) 
;

\draw[bluereadable, line width=0.8mm] (4) -- (-2.5,0);
\draw[bluereadable, line width=0.8mm] (1) -- (-2.5,1);
\draw[bluereadable, line width=0.8mm] (6) -- (2.5,0);
\draw[bluereadable, line width=0.8mm] (3) -- (2.5,1);

\end{tikzpicture}
\vspace{-5.5mm}
    \caption{Case 7}
    \label{fig:case2-1-7}
    \end{subfigure}
     \hfill
     \begin{subfigure}[b]{0.3\linewidth}
    \begin{tikzpicture}[main_node/.style={circle,draw,minimum size=1em,inner sep=2pt},scale=0.8]

\node[main_node] (4) at (-1.5, 0) {};
\node[main_node] (1) at (-1.5, 1) {};
\node[main_node] (5) at (0, 0) {};
\node[main_node,fill=bluereadable] (2) at (0, 1) {};
\node[main_node,fill=bluereadable] (6) at (1.5,0) {};
\node[main_node] (3) at (1.5, 1) {};

 \path[draw, thick]
(4) edge node {} (5) 
(5) edge node {} (6) 
(6) edge node {} (3) 
(3) edge node {} (2) 
(2) edge node {} (1) 
(4) edge node {} (2) 
(1) edge node {} (5) 
;

 \path[draw, line width=0.8mm, color=bluereadable]
(3) edge node {} (6) 
(4) edge node {} (5) 
(1) edge node {} (5)
(2) edge node {} (4) 
;

\draw[gray, thick] (4) -- (-2.5,0);
\draw[bluereadable, line width=0.8mm] (1) -- (-2.5,1);
\draw[gray, thick] (6) -- (2.5,0);
\draw[bluereadable, line width=0.8mm] (3) -- (2.5,1);

\end{tikzpicture}
\vspace{-5.5mm}
    \caption{Case 8}
    \label{fig:case2-1-8}
    \end{subfigure}
     \hfill
    \begin{subfigure}[b]{0.3\linewidth}
    \begin{tikzpicture}[main_node/.style={circle,draw,minimum size=1em,inner sep=2pt},scale=0.8]

\node[main_node] (4) at (-1.5, 0) {};
\node[main_node] (1) at (-1.5, 1) {};
\node[main_node] (5) at (0, 0) {};
\node[main_node] (2) at (0, 1) {};
\node[main_node,fill=bluereadable] (6) at (1.5,0) {};
\node[main_node,fill=bluereadable] (3) at (1.5, 1) {};

 \path[draw, thick]
(4) edge node {} (5) 
(5) edge node {} (6) 
(6) edge node {} (3) 
(3) edge node {} (2) 
(2) edge node {} (1) 
(4) edge node {} (2) 
(1) edge node {} (5) 
;

 \path[draw, line width=0.8mm, color=bluereadable]
(3) edge node {} (2) 
(2) edge node {} (1) 
(1) edge node {} (5)
(5) edge node {} (4) 
;

\draw[bluereadable, line width=0.8mm] (4) -- (-2.5,0);
\draw[gray, thick] (1) -- (-2.5,1);
\draw[bluereadable, line width=0.8mm] (6) -- (2.5,0);
\draw[gray, thick] (3) -- (2.5,1);

\end{tikzpicture}
\vspace{-5.5mm}
    \caption{Case 9}
    \label{fig:case2-1-9}
    \end{subfigure}
    \hfill
    
    \caption{The nine distinct ways to select two vertices in the same copy of $G_A$ in the graph $F_k$ and two subpaths for each such case.}
    \label{fig:intragadget}
    \vspace{-5mm}
\end{figure}
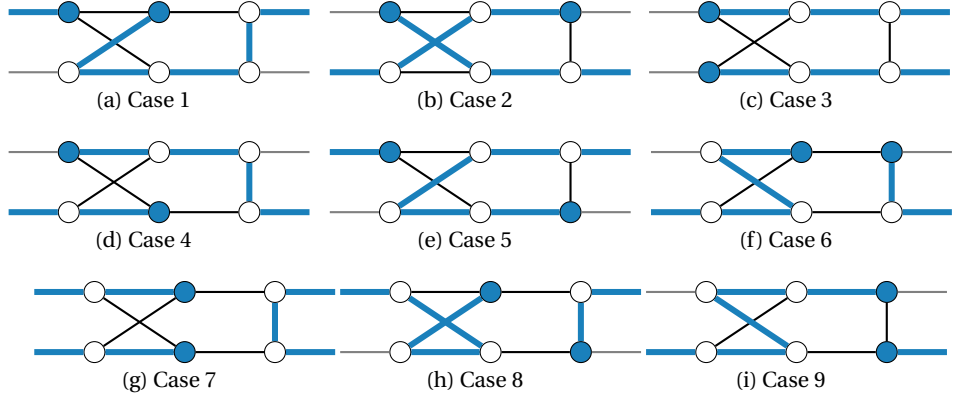

\begin{figure}
    \centering
     \begin{subfigure}[b]{0.3\linewidth}
    \begin{tikzpicture}[main_node/.style={circle,draw,minimum size=1em,inner sep=2pt},scale=0.8]

\node[main_node] (4) at (-1.5, 0) {4};
\node[main_node] (1) at (-1.5, 1) {1};
\node[main_node] (5) at (0, 0) {5};
\node[main_node] (2) at (0, 1) {2};
\node[main_node] (6) at (1.5,0) {6};
\node[main_node] (3) at (1.5, 1) {3};

 \path[draw, thick]
(4) edge node {} (5) 
(5) edge node {} (6) 
(6) edge node {} (3) 
(3) edge node {} (2) 
(2) edge node {} (1) 
(4) edge node {} (2) 
(1) edge node {} (5) 
;

 \path[draw, line width=0.8mm, color=bluereadable]
(1) edge node {} (2) 
(2) edge node {} (4) 
(4) edge node {} (5)
(5) edge node {} (6)
(6) edge node {} (3) 
;

\draw[gray, thick] (4) -- (-2.5,0);
\draw[bluereadable, line width=0.8mm] (1) -- (-2.5,1);
\draw[gray, thick] (6) -- (2.5,0);
\draw[bluereadable, line width=0.8mm] (3) -- (2.5,1);

\end{tikzpicture}
    \caption{Case 1}
    \label{fig:case2-2-1}
    \end{subfigure}
     \hfill
     \begin{subfigure}[b]{0.3\linewidth}
    \begin{tikzpicture}[main_node/.style={circle,draw,minimum size=1em,inner sep=2pt},scale=0.8]

\node[main_node] (4) at (-1.5, 0) {4};
\node[main_node] (1) at (-1.5, 1) {1};
\node[main_node] (5) at (0, 0) {5};
\node[main_node] (2) at (0, 1) {2};
\node[main_node] (6) at (1.5,0) {6};
\node[main_node] (3) at (1.5, 1) {3};

 \path[draw, thick]
(4) edge node {} (5) 
(5) edge node {} (6) 
(6) edge node {} (3) 
(3) edge node {} (2) 
(2) edge node {} (1) 
(4) edge node {} (2) 
(1) edge node {} (5) 
;
 \path[draw, line width=0.8mm, color=bluereadable]
(1) edge node {} (2) 
(2) edge node {} (3) 
(4) edge node {} (5)
(5) edge node {} (6) 
;

\draw[bluereadable, line width=0.8mm] (4) -- (-2.5,0);
\draw[bluereadable, line width=0.8mm] (1) -- (-2.5,1);
\draw[bluereadable, line width=0.8mm] (6) -- (2.5,0);
\draw[bluereadable, line width=0.8mm] (3) -- (2.5,1);

\end{tikzpicture}
    \caption{Case 2}
    \label{fig:case2-2-2}
    \end{subfigure}
    \hfill
    \begin{subfigure}[b]{0.28\linewidth}
    \begin{tikzpicture}[main_node/.style={circle,draw,minimum size=1em,inner sep=2pt},scale=0.8]

\node[main_node] (4) at (-1.5, 0) {4};
\node[main_node] (1) at (-1.5, 1) {1};
\node[main_node] (5) at (0, 0) {5};
\node[main_node] (2) at (0, 1) {2};
\node[main_node] (6) at (1.5,0) {6};
\node[main_node] (3) at (1.5, 1) {3};

 \path[draw, thick]
(4) edge node {} (5) 
(5) edge node {} (6) 
(6) edge node {} (3) 
(3) edge node {} (2) 
(2) edge node {} (1) 
(4) edge node {} (2) 
(1) edge node {} (5) 
;

 \path[draw, line width=0.8mm, color=bluereadable]
(1) edge node {} (2) 
(2) edge node {} (3) 
(4) edge node {} (5)
(5) edge node {} (6)
(3) edge node {} (6) 
;

\draw[bluereadable, line width=0.8mm] (4) -- (-2.5,0);
\draw[bluereadable, line width=0.8mm] (1) -- (-2.5,1);
\draw[gray, thick] (6) -- (2.5,0);
\draw[gray, thick] (3) -- (2.5,1);
\end{tikzpicture}
    \caption{Case 3}
    \label{fig:case2-2-3}
    \end{subfigure}
    \hfill
    \caption{The three sets of subpaths.}
    \label{fig:intragadget-rest}
    \vspace{-5mm}
\end{figure}
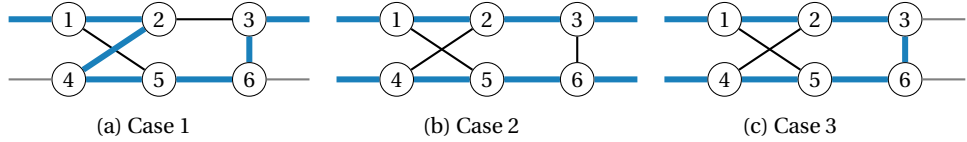

Now suppose that $u=v_{i,p}$ and $w=v_{j,q}$, where $1 \leq i < j \leq k$ and $1 \leq p,q \leq 6$. We will show that $F_k$ contains a hamiltonian $uw$-path by induction on $k$. If $k=2$, there are twelve distinct cases (up to symmetry) and the corresponding hamiltonian paths are shown in Fig.~\ref{fig:intergadget}.

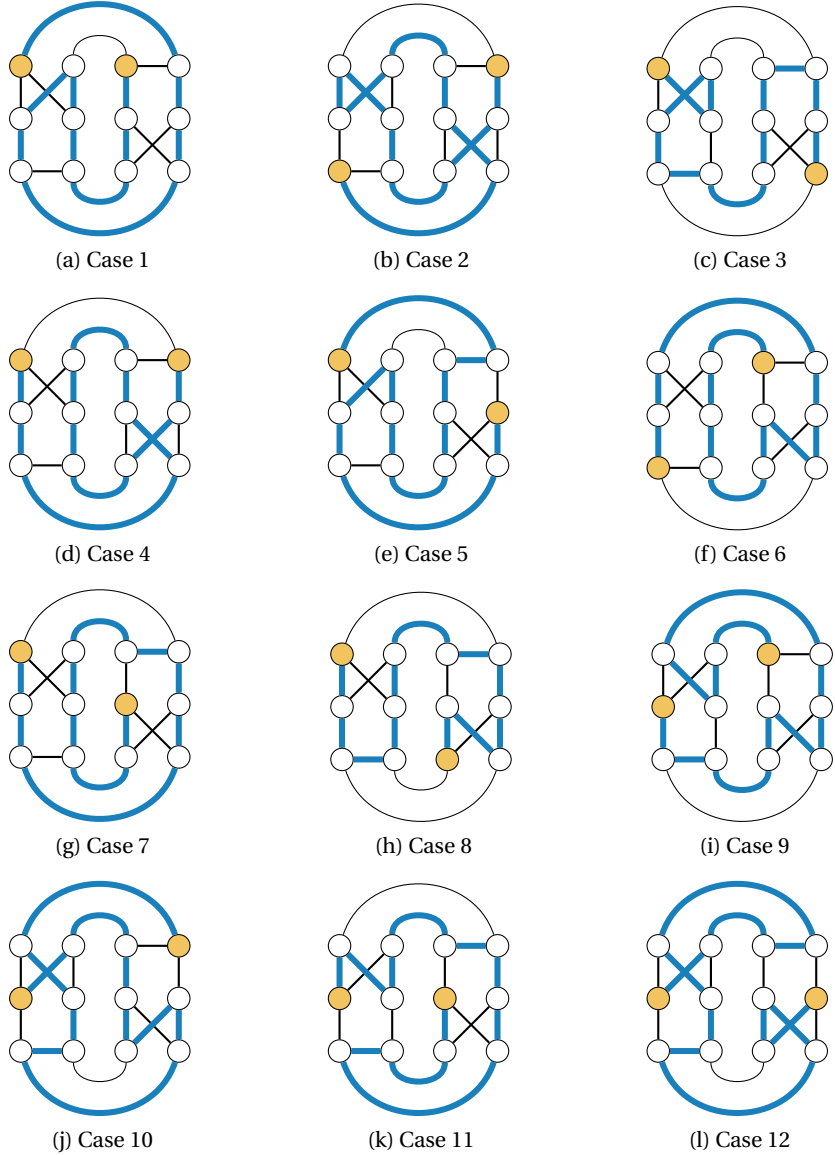
\begin{figure}[h!]
    \centering
     \begin{subfigure}[b]{0.3\textwidth}
     \centering
    \begin{tikzpicture}[main_node/.style={circle,draw,minimum size=1em,inner sep=3pt},scale=0.70]

\node[main_node] (16) at (-1.5, 0) {};
\node[main_node] (13) at (-0.5, 0) {};
\node[main_node] (15) at (-1.5, 1) {};
\node[main_node,fill=redreadable] (14) at (-1.5,2) {};
\node[main_node] (12) at (-0.5, 1) {};
\node[main_node] (11) at (-0.5, 2) {};

 \path[draw, thick]
(14) edge node {} (15) 
(15) edge node {} (16) 
(16) edge node {} (13) 
(13) edge node {} (12) 
(12) edge node {} (11) 
(14) edge node {} (12) 
(11) edge node {} (15)
;
\node[main_node] (21) at (0.5, 0) {};
\node[main_node] (24) at (1.5, 0) {};
\node[main_node] (22) at (0.5, 1) {};
\node[main_node,fill=redreadable] (23) at (0.5,2) {};
\node[main_node] (25) at (1.5, 1) {};
\node[main_node] (26) at (1.5, 2) {};

 \path[draw, thick]
(24) edge node {} (25) 
(25) edge node {} (26) 
(26) edge node {} (23) 
(23) edge node {} (22) 
(22) edge node {} (21) 
(24) edge node {} (22) 
(21) edge node {} (25)
;

 \path[draw, line width=0.8mm, color=bluereadable]
(11) edge node {} (15) 
(11) edge node {} (12) 
(12) edge node {} (13) 
(15) edge node {} (16) 
(21) edge node {} (22) 
(22) edge node {} (23) 
(24) edge node {} (25) 
(25) edge node {} (26) 
;
\draw (14)[line width=0.8mm, color=bluereadable] .. controls (-1,3.5) and (1,3.5) .. (26);
\draw (11) .. controls (-0.5,2.7) and (0.5,2.7) .. (23);
\draw[line width=0.8mm, color=bluereadable] (16) .. controls (-1,-1.5) and (1,-1.5) .. (24);
\draw[line width=0.8mm, color=bluereadable] (13) .. controls (-0.5,-0.7) and (0.5,-0.7) .. (21);

\end{tikzpicture}
\vspace{-3mm}
    \caption{Case 1}
    \label{fig:case2-3-1}
    \end{subfigure}
     \hfill
     \begin{subfigure}[b]{0.3\textwidth}      \centering
     \centering
    \begin{tikzpicture}[main_node/.style={circle,draw,minimum size=1em,inner sep=3pt},scale=0.70]

\node[main_node,fill=redreadable] (16) at (-1.5, 0) {};
\node[main_node] (13) at (-0.5, 0) {};
\node[main_node] (15) at (-1.5, 1) {};
\node[main_node] (14) at (-1.5,2) {};
\node[main_node] (12) at (-0.5, 1) {};
\node[main_node] (11) at (-0.5, 2) {};

 \path[draw, thick]
(14) edge node {} (15) 
(15) edge node {} (16) 
(16) edge node {} (13) 
(13) edge node {} (12) 
(12) edge node {} (11) 
(14) edge node {} (12) 
(11) edge node {} (15)
;
\node[main_node] (21) at (0.5, 0) {};
\node[main_node] (24) at (1.5, 0) {};
\node[main_node] (22) at (0.5, 1) {};
\node[main_node] (23) at (0.5,2) {};
\node[main_node] (25) at (1.5, 1) {};
\node[main_node,fill=redreadable] (26) at (1.5, 2) {};

 \path[draw, thick]
(24) edge node {} (25) 
(25) edge node {} (26) 
(26) edge node {} (23) 
(23) edge node {} (22) 
(22) edge node {} (21) 
(24) edge node {} (22) 
(21) edge node {} (25)
;

 \path[draw, line width=0.8mm, color=bluereadable]
(11) edge node {} (15) 
(15) edge node {} (14) 
(14) edge node {} (12) 
(12) edge node {} (13) 
(21) edge node {} (25) 
(25) edge node {} (26) 
(24) edge node {} (22) 
(22) edge node {} (23) 
;
\draw (14) .. controls (-1,3.5) and (1,3.5) .. (26);
\draw (11)[line width=0.8mm, color=bluereadable] .. controls (-0.5,2.7) and (0.5,2.7) .. (23);
\draw[line width=0.8mm, color=bluereadable] (16) .. controls (-1,-1.5) and (1,-1.5) .. (24);
\draw[line width=0.8mm, color=bluereadable] (13) .. controls (-0.5,-0.7) and (0.5,-0.7) .. (21);

\end{tikzpicture}
\vspace{-3mm}
    \caption{Case 2}
    \label{fig:case2-3-2}
    \end{subfigure}
    \hfill
    \begin{subfigure}[b]{0.3\textwidth}      \centering
    \begin{tikzpicture}[main_node/.style={circle,draw,minimum size=1em,inner sep=3pt},scale=0.70]

\node[main_node] (16) at (-1.5, 0) {};
\node[main_node] (13) at (-0.5, 0) {};
\node[main_node] (15) at (-1.5, 1) {};
\node[main_node,fill=redreadable] (14) at (-1.5,2) {};
\node[main_node] (12) at (-0.5, 1) {};
\node[main_node] (11) at (-0.5, 2) {};

 \path[draw, thick]
(14) edge node {} (15) 
(15) edge node {} (16) 
(16) edge node {} (13) 
(13) edge node {} (12) 
(12) edge node {} (11) 
(14) edge node {} (12) 
(11) edge node {} (15)
;
\node[main_node] (21) at (0.5, 0) {};
\node[main_node,fill=redreadable] (24) at (1.5, 0) {};
\node[main_node] (22) at (0.5, 1) {};
\node[main_node] (23) at (0.5,2) {};
\node[main_node] (25) at (1.5, 1) {};
\node[main_node] (26) at (1.5, 2) {};

 \path[draw, thick]
(24) edge node {} (25) 
(25) edge node {} (26) 
(26) edge node {} (23) 
(23) edge node {} (22) 
(22) edge node {} (21) 
(24) edge node {} (22) 
(21) edge node {} (25)
;

  \path[draw, line width=0.8mm, color=bluereadable]
(11) edge node {} (15) 
(11) edge node {} (12) 
(14) edge node {} (12) 
(15) edge node {} (16)
(13) edge node {} (16) 
(21) edge node {} (22) 
(22) edge node {} (23) 
(23) edge node {} (26) 
(26) edge node {} (25)
(25) edge node {} (24)
;
\draw (14) .. controls (-1,3.5) and (1,3.5) .. (26);
\draw (11) .. controls (-0.5,2.7) and (0.5,2.7) .. (23);
\draw (16) .. controls (-1,-1.5) and (1,-1.5) .. (24);
\draw[line width=0.8mm, color=bluereadable] (13) .. controls (-0.5,-0.7) and (0.5,-0.7) .. (21);

\end{tikzpicture}
\vspace{-3mm}
    \caption{Case 3}
    \label{fig:case2-3-3}
    \end{subfigure}
    \hfill

\centering
     \begin{subfigure}[b]{0.3\textwidth}      \centering
    \begin{tikzpicture}[main_node/.style={circle,draw,minimum size=1em,inner sep=3pt},scale=0.70]

\node[main_node] (16) at (-1.5, 0) {};
\node[main_node] (13) at (-0.5, 0) {};
\node[main_node] (15) at (-1.5, 1) {};
\node[main_node,fill=redreadable] (14) at (-1.5,2) {};
\node[main_node] (12) at (-0.5, 1) {};
\node[main_node] (11) at (-0.5, 2) {};

 \path[draw, thick]
(14) edge node {} (15) 
(15) edge node {} (16) 
(16) edge node {} (13) 
(13) edge node {} (12) 
(12) edge node {} (11) 
(14) edge node {} (12) 
(11) edge node {} (15)
;
\node[main_node] (21) at (0.5, 0) {};
\node[main_node] (24) at (1.5, 0) {};
\node[main_node] (22) at (0.5, 1) {};
\node[main_node] (23) at (0.5,2) {};
\node[main_node] (25) at (1.5, 1) {};
\node[main_node,fill=redreadable] (26) at (1.5, 2) {};

 \path[draw, thick]
(24) edge node {} (25) 
(25) edge node {} (26) 
(26) edge node {} (23) 
(23) edge node {} (22) 
(22) edge node {} (21) 
(24) edge node {} (22) 
(21) edge node {} (25)
;

  \path[draw, line width=0.8mm, color=bluereadable]
(11) edge node {} (12) 
(13) edge node {} (12) 
(14) edge node {} (15) 
(15) edge node {} (16)
(26) edge node {} (25) 
(25) edge node {} (21) 
(23) edge node {} (22) 
(22) edge node {} (24)
;
\draw (14) .. controls (-1,3.5) and (1,3.5) .. (26);
\draw (11)[line width=0.8mm, color=bluereadable] .. controls (-0.5,2.7) and (0.5,2.7) .. (23);
\draw (16)[line width=0.8mm, color=bluereadable] .. controls (-1,-1.5) and (1,-1.5) .. (24);
\draw[line width=0.8mm, color=bluereadable] (13) .. controls (-0.5,-0.7) and (0.5,-0.7) .. (21);

\end{tikzpicture}
\vspace{-3mm}
    \caption{Case 4}
    \label{fig:case2-3-4}
    \end{subfigure}
     \hfill
     \begin{subfigure}[b]{0.3\textwidth}      \centering
    \begin{tikzpicture}[main_node/.style={circle,draw,minimum size=1em,inner sep=3pt},scale=0.70]

\node[main_node] (16) at (-1.5, 0) {};
\node[main_node] (13) at (-0.5, 0) {};
\node[main_node] (15) at (-1.5, 1) {};
\node[main_node,fill=redreadable] (14) at (-1.5,2) {};
\node[main_node] (12) at (-0.5, 1) {};
\node[main_node] (11) at (-0.5, 2) {};

 \path[draw, thick]
(14) edge node {} (15) 
(15) edge node {} (16) 
(16) edge node {} (13) 
(13) edge node {} (12) 
(12) edge node {} (11) 
(14) edge node {} (12) 
(11) edge node {} (15)
;
\node[main_node] (21) at (0.5, 0) {};
\node[main_node] (24) at (1.5, 0) {};
\node[main_node] (22) at (0.5, 1) {};
\node[main_node] (23) at (0.5,2) {};
\node[main_node,fill=redreadable] (25) at (1.5, 1) {};
\node[main_node] (26) at (1.5, 2) {};

 \path[draw, thick]
(24) edge node {} (25) 
(25) edge node {} (26) 
(26) edge node {} (23) 
(23) edge node {} (22) 
(22) edge node {} (21) 
(24) edge node {} (22) 
(21) edge node {} (25)
;

\path[draw, line width=0.8mm, color=bluereadable]
(11) edge node {} (12) 
(11) edge node {} (15) 
(15) edge node {} (16) 
(12) edge node {} (13)
(26) edge node {} (23) 
(25) edge node {} (24) 
(23) edge node {} (22) 
(22) edge node {} (21)
;
\draw (14)[line width=0.8mm, color=bluereadable] .. controls (-1,3.5) and (1,3.5) .. (26);
\draw (11) .. controls (-0.5,2.7) and (0.5,2.7) .. (23);
\draw (16)[line width=0.8mm, color=bluereadable] .. controls (-1,-1.5) and (1,-1.5) .. (24);
\draw[line width=0.8mm, color=bluereadable] (13) .. controls (-0.5,-0.7) and (0.5,-0.7) .. (21);

\end{tikzpicture}
\vspace{-3mm}
    \caption{Case 5}
    \label{fig:case2-3-5}
    \end{subfigure}
    \hfill
    \begin{subfigure}[b]{0.3\textwidth}      \centering
    \begin{tikzpicture}[main_node/.style={circle,draw,minimum size=1em,inner sep=3pt},scale=0.70]

\node[main_node,fill=redreadable] (16) at (-1.5, 0) {};
\node[main_node] (13) at (-0.5, 0) {};
\node[main_node] (15) at (-1.5, 1) {};
\node[main_node] (14) at (-1.5,2) {};
\node[main_node] (12) at (-0.5, 1) {};
\node[main_node] (11) at (-0.5, 2) {};

 \path[draw, thick]
(14) edge node {} (15) 
(15) edge node {} (16) 
(16) edge node {} (13) 
(13) edge node {} (12) 
(12) edge node {} (11) 
(14) edge node {} (12) 
(11) edge node {} (15)
;
\node[main_node] (21) at (0.5, 0) {};
\node[main_node] (24) at (1.5, 0) {};
\node[main_node] (22) at (0.5, 1) {};
\node[main_node,fill=redreadable] (23) at (0.5,2) {};
\node[main_node] (25) at (1.5, 1) {};
\node[main_node] (26) at (1.5, 2) {};

 \path[draw, thick]
(24) edge node {} (25) 
(25) edge node {} (26) 
(26) edge node {} (23) 
(23) edge node {} (22) 
(22) edge node {} (21) 
(24) edge node {} (22) 
(21) edge node {} (25)
;

\path[draw, line width=0.8mm, color=bluereadable]
(11) edge node {} (12) 
(12) edge node {} (13) 
(14) edge node {} (15) 
(15) edge node {} (16)
(21) edge node {} (22) 
(22) edge node {} (24) 
(24) edge node {} (25) 
(25) edge node {} (26)
;
\draw (14)[line width=0.8mm, color=bluereadable] .. controls (-1,3.5) and (1,3.5) .. (26);
\draw (11)[line width=0.8mm, color=bluereadable] .. controls (-0.5,2.7) and (0.5,2.7) .. (23);
\draw (16) .. controls (-1,-1.5) and (1,-1.5) .. (24);
\draw[line width=0.8mm, color=bluereadable] (13) .. controls (-0.5,-0.7) and (0.5,-0.7) .. (21);

\end{tikzpicture}
\vspace{-3mm}
    \caption{Case 6}
    \label{fig:case2-3-6}
    \end{subfigure}
    \hfill

    \centering
     \begin{subfigure}[b]{0.3\textwidth}      \centering
    \begin{tikzpicture}[main_node/.style={circle,draw,minimum size=1em,inner sep=3pt},scale=0.70]

\node[main_node] (16) at (-1.5, 0) {};
\node[main_node] (13) at (-0.5, 0) {};
\node[main_node] (15) at (-1.5, 1) {};
\node[main_node,fill=redreadable] (14) at (-1.5,2) {};
\node[main_node] (12) at (-0.5, 1) {};
\node[main_node] (11) at (-0.5, 2) {};

 \path[draw, thick]
(14) edge node {} (15) 
(15) edge node {} (16) 
(16) edge node {} (13) 
(13) edge node {} (12) 
(12) edge node {} (11) 
(14) edge node {} (12) 
(11) edge node {} (15)
;
\node[main_node] (21) at (0.5, 0) {};
\node[main_node] (24) at (1.5, 0) {};
\node[main_node,fill=redreadable] (22) at (0.5, 1) {};
\node[main_node] (23) at (0.5,2) {};
\node[main_node] (25) at (1.5, 1) {};
\node[main_node] (26) at (1.5, 2) {};

 \path[draw, thick]
(24) edge node {} (25) 
(25) edge node {} (26) 
(26) edge node {} (23) 
(23) edge node {} (22) 
(22) edge node {} (21) 
(24) edge node {} (22) 
(21) edge node {} (25)
;

\path[draw, line width=0.8mm, color=bluereadable]
(11) edge node {} (12) 
(12) edge node {} (13) 
(14) edge node {} (15) 
(15) edge node {} (16)
(21) edge node {} (22) 
(23) edge node {} (26) 
(24) edge node {} (25) 
(25) edge node {} (26)
;
\draw (14) .. controls (-1,3.5) and (1,3.5) .. (26);
\draw (11)[line width=0.8mm, color=bluereadable] .. controls (-0.5,2.7) and (0.5,2.7) .. (23);
\draw (16) [line width=0.8mm, color=bluereadable] .. controls (-1,-1.5) and (1,-1.5) .. (24);
\draw[line width=0.8mm, color=bluereadable] (13) .. controls (-0.5,-0.7) and (0.5,-0.7) .. (21);

\end{tikzpicture}
\vspace{-3mm}
    \caption{Case 7}
    \label{fig:case2-3-7}
    \end{subfigure}
     \hfill
\begin{subfigure}[b]{0.3\textwidth}      \centering
    \begin{tikzpicture}[main_node/.style={circle,draw,minimum size=1em,inner sep=3pt},scale=0.70]

\node[main_node] (16) at (-1.5, 0) {};
\node[main_node] (13) at (-0.5, 0) {};
\node[main_node] (15) at (-1.5, 1) {};
\node[main_node,fill=redreadable] (14) at (-1.5,2) {};
\node[main_node] (12) at (-0.5, 1) {};
\node[main_node] (11) at (-0.5, 2) {};

 \path[draw, thick]
(14) edge node {} (15) 
(15) edge node {} (16) 
(16) edge node {} (13) 
(13) edge node {} (12) 
(12) edge node {} (11) 
(14) edge node {} (12) 
(11) edge node {} (15)
;
\node[main_node,fill=redreadable] (21) at (0.5, 0) {};
\node[main_node] (24) at (1.5, 0) {};
\node[main_node] (22) at (0.5, 1) {};
\node[main_node] (23) at (0.5,2) {};
\node[main_node] (25) at (1.5, 1) {};
\node[main_node] (26) at (1.5, 2) {};

 \path[draw, thick]
(24) edge node {} (25) 
(25) edge node {} (26) 
(26) edge node {} (23) 
(23) edge node {} (22) 
(22) edge node {} (21) 
(24) edge node {} (22) 
(21) edge node {} (25)
;

\path[draw, line width=0.8mm, color=bluereadable]
(11) edge node {} (12) 
(12) edge node {} (13) 
(13) edge node {} (16) 
(15) edge node {} (16)
(14) edge node {} (15)
(21) edge node {} (22) 
(22) edge node {} (24) 
(24) edge node {} (25) 
(25) edge node {} (26)
(23) edge node {} (26)
;
\draw (14) .. controls (-1,3.5) and (1,3.5) .. (26);
\draw (11)[line width=0.8mm, color=bluereadable] .. controls (-0.5,2.7) and (0.5,2.7) .. (23);
\draw (16) .. controls (-1,-1.5) and (1,-1.5) .. (24);
\draw (13) .. controls (-0.5,-0.7) and (0.5,-0.7) .. (21);

\end{tikzpicture}
\vspace{-3mm}
    \caption{Case 8}
    \label{fig:case2-3-8}
    \end{subfigure}
    \hfill
    \begin{subfigure}[b]{0.3\textwidth}      \centering
    \begin{tikzpicture}[main_node/.style={circle,draw,minimum size=1em,inner sep=3pt},scale=0.70]

\node[main_node] (16) at (-1.5, 0) {};
\node[main_node] (13) at (-0.5, 0) {};
\node[main_node,fill=redreadable] (15) at (-1.5, 1) {};
\node[main_node] (14) at (-1.5,2) {};
\node[main_node] (12) at (-0.5, 1) {};
\node[main_node] (11) at (-0.5, 2) {};

 \path[draw, thick]
(14) edge node {} (15) 
(15) edge node {} (16) 
(16) edge node {} (13) 
(13) edge node {} (12) 
(12) edge node {} (11) 
(14) edge node {} (12) 
(11) edge node {} (15)
;
\node[main_node] (21) at (0.5, 0) {};
\node[main_node] (24) at (1.5, 0) {};
\node[main_node] (22) at (0.5, 1) {};
\node[main_node,fill=redreadable] (23) at (0.5,2) {};
\node[main_node] (25) at (1.5, 1) {};
\node[main_node] (26) at (1.5, 2) {};

 \path[draw, thick]
(24) edge node {} (25) 
(25) edge node {} (26) 
(26) edge node {} (23) 
(23) edge node {} (22) 
(22) edge node {} (21) 
(24) edge node {} (22) 
(21) edge node {} (25)
;

\path[draw, line width=0.8mm, color=bluereadable]
(11) edge node {} (12) 
(12) edge node {} (14) 
(15) edge node {} (16) 
(16) edge node {} (13)
(21) edge node {} (22) 
(22) edge node {} (24) 
(24) edge node {} (25) 
(25) edge node {} (26)
;
\draw (14)[line width=0.8mm, color=bluereadable] .. controls (-1,3.5) and (1,3.5) .. (26);
\draw (11)[line width=0.8mm, color=bluereadable] .. controls (-0.5,2.7) and (0.5,2.7) .. (23);
\draw (16) .. controls (-1,-1.5) and (1,-1.5) .. (24);
\draw (13)[line width=0.8mm, color=bluereadable] .. controls (-0.5,-0.7) and (0.5,-0.7) .. (21);

\end{tikzpicture}
\vspace{-3mm}
    \caption{Case 9}
    \label{fig:case2-3-9}
\end{subfigure}
\hfill
\begin{subfigure}[b]{0.3\textwidth}      \centering
    \begin{tikzpicture}[main_node/.style={circle,draw,minimum size=1em,inner sep=3pt},scale=0.70]

\node[main_node] (16) at (-1.5, 0) {};
\node[main_node] (13) at (-0.5, 0) {};
\node[main_node,fill=redreadable] (15) at (-1.5, 1) {};
\node[main_node] (14) at (-1.5,2) {};
\node[main_node] (12) at (-0.5, 1) {};
\node[main_node] (11) at (-0.5, 2) {};

 \path[draw, thick]
(14) edge node {} (15) 
(15) edge node {} (16) 
(16) edge node {} (13) 
(13) edge node {} (12) 
(12) edge node {} (11) 
(14) edge node {} (12) 
(11) edge node {} (15)
;
\node[main_node] (21) at (0.5, 0) {};
\node[main_node] (24) at (1.5, 0) {};
\node[main_node] (22) at (0.5, 1) {};
\node[main_node] (23) at (0.5,2) {};
\node[main_node] (25) at (1.5, 1) {};
\node[main_node,fill=redreadable] (26) at (1.5, 2) {};

 \path[draw, thick]
(24) edge node {} (25) 
(25) edge node {} (26) 
(26) edge node {} (23) 
(23) edge node {} (22) 
(22) edge node {} (21) 
(24) edge node {} (22) 
(21) edge node {} (25)
;

\path[draw, line width=0.8mm, color=bluereadable]
(11) edge node {} (15) 
(12) edge node {} (14) 
(12) edge node {} (13) 
(16) edge node {} (13)
(21) edge node {} (22) 
(22) edge node {} (23) 
(21) edge node {} (25) 
(25) edge node {} (24)
;
\draw (14)[line width=0.8mm, color=bluereadable] .. controls (-1,3.5) and (1,3.5) .. (26);
\draw (11)[line width=0.8mm, color=bluereadable] .. controls (-0.5,2.7) and (0.5,2.7) .. (23);
\draw (16)[line width=0.8mm, color=bluereadable] .. controls (-1,-1.5) and (1,-1.5) .. (24);
\draw (13) .. controls (-0.5,-0.7) and (0.5,-0.7) .. (21);

\end{tikzpicture}
\vspace{-3mm}
    \caption{Case 10}
    \label{fig:case2-3-10}
\end{subfigure}
\hfill
\begin{subfigure}[b]{0.3\textwidth}      \centering
    \begin{tikzpicture}[main_node/.style={circle,draw,minimum size=1em,inner sep=3pt},scale=0.70]

\node[main_node] (16) at (-1.5, 0) {};
\node[main_node] (13) at (-0.5, 0) {};
\node[main_node,fill=redreadable] (15) at (-1.5, 1) {};
\node[main_node] (14) at (-1.5,2) {};
\node[main_node] (12) at (-0.5, 1) {};
\node[main_node] (11) at (-0.5, 2) {};

 \path[draw, thick]
(14) edge node {} (15) 
(15) edge node {} (16) 
(16) edge node {} (13) 
(13) edge node {} (12) 
(12) edge node {} (11) 
(14) edge node {} (12) 
(11) edge node {} (15)
;
\node[main_node] (21) at (0.5, 0) {};
\node[main_node] (24) at (1.5, 0) {};
\node[main_node,fill=redreadable] (22) at (0.5, 1) {};
\node[main_node] (23) at (0.5,2) {};
\node[main_node] (25) at (1.5, 1) {};
\node[main_node] (26) at (1.5, 2) {};

 \path[draw, thick]
(24) edge node {} (25) 
(25) edge node {} (26) 
(26) edge node {} (23) 
(23) edge node {} (22) 
(22) edge node {} (21) 
(24) edge node {} (22) 
(21) edge node {} (25)
;

\path[draw, line width=0.8mm, color=bluereadable]
(11) edge node {} (12) 
(12) edge node {} (14) 
(14) edge node {} (15) 
(16) edge node {} (13)
(21) edge node {} (22) 
(26) edge node {} (23) 
(26) edge node {} (25) 
(25) edge node {} (24)
;
\draw (14) .. controls (-1,3.5) and (1,3.5) .. (26);
\draw (11)[line width=0.8mm, color=bluereadable] .. controls (-0.5,2.7) and (0.5,2.7) .. (23);
\draw (16)[line width=0.8mm, color=bluereadable] .. controls (-1,-1.5) and (1,-1.5) .. (24);
\draw (13)[line width=0.8mm, color=bluereadable] .. controls (-0.5,-0.7) and (0.5,-0.7) .. (21);

\end{tikzpicture}
\vspace{-3mm}
    \caption{Case 11}
    \label{fig:case2-3-11}
\end{subfigure}
\hfill
\begin{subfigure}[b]{0.3\textwidth}      \centering
    \begin{tikzpicture}[main_node/.style={circle,draw,minimum size=1em,inner sep=3pt},scale=0.70]

\node[main_node] (16) at (-1.5, 0) {};
\node[main_node] (13) at (-0.5, 0) {};
\node[main_node,fill=redreadable] (15) at (-1.5, 1) {};
\node[main_node] (14) at (-1.5,2) {};
\node[main_node] (12) at (-0.5, 1) {};
\node[main_node] (11) at (-0.5, 2) {};

 \path[draw, thick]
(14) edge node {} (15) 
(15) edge node {} (16) 
(16) edge node {} (13) 
(13) edge node {} (12) 
(12) edge node {} (11) 
(14) edge node {} (12) 
(11) edge node {} (15)
;
\node[main_node] (21) at (0.5, 0) {};
\node[main_node] (24) at (1.5, 0) {};
\node[main_node] (22) at (0.5, 1) {};
\node[main_node] (23) at (0.5,2) {};
\node[main_node,fill=redreadable] (25) at (1.5, 1) {};
\node[main_node] (26) at (1.5, 2) {};

 \path[draw, thick]
(24) edge node {} (25) 
(25) edge node {} (26) 
(26) edge node {} (23) 
(23) edge node {} (22) 
(22) edge node {} (21) 
(24) edge node {} (22) 
(21) edge node {} (25)
;

\path[draw, line width=0.8mm, color=bluereadable]
(11) edge node {} (15) 
(12) edge node {} (14) 
(12) edge node {} (13) 
(16) edge node {} (13)
(21) edge node {} (22) 
(26) edge node {} (23) 
(22) edge node {} (24) 
(25) edge node {} (21)
;
\draw (14)[line width=0.8mm, color=bluereadable] .. controls (-1,3.5) and (1,3.5) .. (26);
\draw (11)[line width=0.8mm, color=bluereadable] .. controls (-0.5,2.7) and (0.5,2.7) .. (23);
\draw (16)[line width=0.8mm, color=bluereadable] .. controls (-1,-1.5) and (1,-1.5) .. (24);
\draw (13) .. controls (-0.5,-0.7) and (0.5,-0.7) .. (21);

\end{tikzpicture}
\vspace{-3mm}
    \caption{Case 12}
    \label{fig:case2-3-12}
\end{subfigure}
\hfill
    \caption{The twelve distinct ways to select two vertices in different copies of $G_A$.}
    \label{fig:intergadget}
\end{figure}

Now suppose that $F_h$ contains a hamiltonian $uw$-path for all distinct $u, w \in V(F_h)$ belonging to different copies of the graph $G_A$, for each $h \in [2, k-1]$. The strategy to obtain a hamiltonian $uw$-path in $F_k$ is to take a hamiltonian path in $F_{k-1}$ and modify it.

More precisely, note that $F_k$ can be obtained by taking $F_{k-1}$, removing the edges $v_{k-1,3}v_{1,1}$ and $v_{k-1,6}v_{1,4}$, adding a new copy of the graph $G_A$ (whose vertices are labelled as $v_{k,1}, v_{k,2}, \ldots, v_{k,6}$) and adding the edges $v_{k-1,3}v_{k,1}$, $v_{k-1,6}v_{k,4}$, $v_{k,3}v_{1,1}$ and $v_{k,6}v_{1,4}$. In the remainder of this proof, we see $F_{k-1} - v_{k-1,3}v_{1,1} - v_{k-1,6}v_{1,4}$ as a subgraph of $F_k$. We first show the existence of a hamiltonian $uw$-path in $F_k$ for all distinct $u, w \in V(F_{k-1})$ which are in different copies of the graph $G_A$. Let $P$ be a hamiltonian $uw$-path in $F_{k-1}$. There are now three cases. 

\smallskip

\noindent \textsc{Case 1}: $E(P)$ contains $v_{k-1,3}v_{1,1}$, but does not contain $v_{k-1,6}v_{1,4}$ (or vice-versa). Then $P$ can be modified to obtain a hamiltonian $uw$-path in $F_k$ by replacing the edge $v_{k-1,3}v_{1,1}$ (or $v_{k-1,6}v_{1,4}$, respectively) by the subpath shown in Fig.~\ref{fig:case2-2-1} in copy $k$ of the graph $G_A$. 

\smallskip

\noindent \textsc{Case 2}: $E(P)$ contains both $v_{k-1,3}v_{1,1}$ and $v_{k-1,6}v_{1,4}$. Then $P$ can be modified to obtain a hamiltonian $uw$-path in $F_k$ by replacing the edges $v_{k-1,3}v_{1,1}$ and $v_{k-1,6}v_{1,4}$ by the subpaths shown in Fig.~\ref{fig:case2-2-2} in copy $k$ of the graph $G_A$. 

\smallskip

\noindent \textsc{Case 3}: $E(P)$ contains neither $v_{k-1,3}v_{1,1}$ nor $v_{k-1,6}v_{1,4}$. Since the graph is cubic and $u$ and $w$ are in different copies of the graph $G_A$, the edge $v_{k-1,3}v_{k-1,6}$ must be in $E(P)$. Then $P$ can be modified to obtain a hamiltonian $uw$-path in $F_k$ by removing the edge $v_{k-1,3}v_{k-1,6}$ (thereby creating two subpaths) and then using the subpath shown in Fig.~\ref{fig:case2-2-3} in copy $k$ of the graph $G_A$.

\smallskip

Finally, whenever $u$ or $w$ is in $V(F_k) \setminus V(F_{k-1})$ there exists an isomorphism $\phi : V(F_k) \rightarrow V(F_k)$ for which $\phi(u) \in V(F_{k-1})$ and $\phi(w) \in V(F_{k-1})$. Therefore, in this case there also exists a hamiltonian $uw$-path in $F_k$.
\end{proof}

We now show that many cycle lengths are missing in $F_k$.

\begin{lemma}
\label{lem:FCycNotPresent}
For each integer $k \geq 2$, $F_k$ does not contain any cycle of length $3 \ell$ for all $\ell \in [k-1]$ or $3 \ell+1$ for all $\ell \in [2, k-1]$.
\end{lemma}

\begin{proof}
The statement can be verified separately for $k=2$ by enumerating all cycles, so henceforth we assume that $k \geq 3$. Consider any cycle $C$ in $F_k$ containing strictly less than $3k$ edges. There must exist a copy of $G_A$ in $F_k$ such that $C$ does not contain any vertices in that copy. Without loss of generality, we may assume that there exist integers $i$ and $j$ such that (i) $1 \leq i \leq j \leq k$; (ii) we do not simultaneously have $i=1$ and $j=k$; and (iii) $C$ contains some vertices of copy $i$ of $G_A$, $C$ contains some vertices of copy $j$ of $G_A$ and $C$ contains all vertices of copy $m$ of $G_A$ for each $i<m<j$ (since there are only two distinct ways in which $C$ can contain some vertices in copy $m$ of $G_A$ as shown in Fig.~\ref{fig:cross-options}).

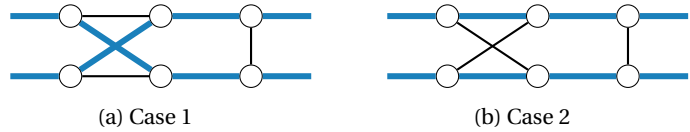
\begin{figure}
\vspace{-1mm}
    \centering
     \begin{subfigure}[b]{0.3\textwidth}
    \begin{tikzpicture}[main_node/.style={circle,draw,minimum size=1em,inner sep=3pt},scale=0.8]

\node[main_node] (4) at (-1.5, 0) {};
\node[main_node] (1) at (-1.5, 1) {};
\node[main_node] (5) at (0, 0) {};
\node[main_node] (2) at (0, 1) {};
\node[main_node] (6) at (1.5,0) {};
\node[main_node] (3) at (1.5, 1) {};

 \path[draw, thick]
(4) edge node {} (5) 
(5) edge node {} (6) 
(6) edge node {} (3) 
(3) edge node {} (2) 
(2) edge node {} (1) 
(4) edge node {} (2) 
(1) edge node {} (5) 
;
 \path[draw, line width=0.8mm, color=bluereadable]
(1) edge node {} (5)
(2) edge node {} (4)
(2) edge node {} (3)
(5) edge node {} (6)
;

\draw[bluereadable, line width=0.8mm] (4) -- (-2.5,0);
\draw[bluereadable, line width=0.8mm] (1) -- (-2.5,1);
\draw[bluereadable, line width=0.8mm] (6) -- (2.5,0);
\draw[bluereadable, line width=0.8mm] (3) -- (2.5,1);

\end{tikzpicture}
    \caption{Case 1}
    \label{fig:case2-6-1}
    \end{subfigure}
    \hspace{0.1\textwidth}
    \begin{subfigure}[b]{0.3\textwidth}
    \begin{tikzpicture}[main_node/.style={circle,draw,minimum size=1em,inner sep=3pt},scale=0.8]

\node[main_node] (4) at (-1.5, 0) {};
\node[main_node] (1) at (-1.5, 1) {};
\node[main_node] (5) at (0, 0) {};
\node[main_node] (2) at (0, 1) {};
\node[main_node] (6) at (1.5,0) {};
\node[main_node] (3) at (1.5, 1) {};

 \path[draw, thick]
(4) edge node {} (5) 
(5) edge node {} (6) 
(6) edge node {} (3) 
(3) edge node {} (2) 
(2) edge node {} (1) 
(4) edge node {} (2) 
(1) edge node {} (5) 
;

 \path[draw, line width=0.8mm, color=bluereadable]
(1) edge node {} (2) 
(2) edge node {} (3) 
(4) edge node {} (5)
(5) edge node {} (6)
;

\draw[bluereadable, line width=0.8mm] (4) -- (-2.5,0);
\draw[bluereadable, line width=0.8mm] (1) -- (-2.5,1);
\draw[bluereadable, line width=0.8mm] (6) -- (2.5,0);
\draw[bluereadable, line width=0.8mm] (3) -- (2.5,1);
\end{tikzpicture}
    \caption{Case 2}
    \label{fig:case2-6-2}
    \end{subfigure}
    \caption{The two distinct cases.}
    \label{fig:cross-options}
    \vspace{-5mm}
\end{figure}

If $i=j$, then $C$ has length 4 or $5$, so in what follows we assume that $i<j$. There are only two distinct ways in which $C$ can contain some vertices in copy $i$ of $G_A$ (in which case it contains either two or five vertices in copy $i$) as shown in Fig.~\ref{fig:start-options}.

\begin{figure}
    \centering
     \begin{subfigure}[b]{0.3\textwidth}
    \begin{tikzpicture}[main_node/.style={circle,draw,minimum size=1em,inner sep=3pt},scale=0.8]

\node[main_node] (4) at (-1.5, 0) {};
\node[main_node] (1) at (-1.5, 1) {};
\node[main_node] (5) at (0, 0) {};
\node[main_node] (2) at (0, 1) {};
\node[main_node] (6) at (1.5,0) {};
\node[main_node] (3) at (1.5, 1) {};

 \path[draw, thick]
(4) edge node {} (5) 
(5) edge node {} (6) 
(6) edge node {} (3) 
(3) edge node {} (2) 
(2) edge node {} (1) 
(4) edge node {} (2) 
(1) edge node {} (5) 
;
 \path[draw, line width=0.8mm, color=bluereadable]
(3) edge node {} (6)
;

\draw[gray, thick] (4) -- (-2.5,0);
\draw[gray, thick] (1) -- (-2.5,1);
\draw[bluereadable, line width=0.8mm] (6) -- (2.5,0);
\draw[bluereadable, line width=0.8mm] (3) -- (2.5,1);

\end{tikzpicture}
    \caption{Case 1}
    \label{fig:case2-5-1}
    \end{subfigure}
    \hspace{0.1\textwidth}
    \begin{subfigure}[b]{0.3\textwidth}
    \begin{tikzpicture}[main_node/.style={circle,draw,minimum size=1em,inner sep=3pt},scale=0.8]

\node[main_node] (4) at (-1.5, 0) {};
\node[main_node] (1) at (-1.5, 1) {};
\node[main_node] (5) at (0, 0) {};
\node[main_node] (2) at (0, 1) {};
\node[main_node] (6) at (1.5,0) {};
\node[main_node] (3) at (1.5, 1) {};

 \path[draw, thick]
(4) edge node {} (5) 
(5) edge node {} (6) 
(6) edge node {} (3) 
(3) edge node {} (2) 
(2) edge node {} (1) 
(4) edge node {} (2) 
(1) edge node {} (5) 
;

 \path[draw, line width=0.8mm, color=bluereadable]
(3) edge node {} (2) 
(2) edge node {} (4) 
(4) edge node {} (5)
(5) edge node {} (6)
;

\draw[gray, thick] (4) -- (-2.5,0);
\draw[gray, thick] (1) -- (-2.5,1);
\draw[bluereadable, line width=0.8mm] (6) -- (2.5,0);
\draw[bluereadable, line width=0.8mm] (3) -- (2.5,1);
\end{tikzpicture}
    \caption{Case 2}
    \label{fig:case2-5-2}
    \end{subfigure}
    \caption{The two distinct cases in which $C$ contains some vertices in copy $i$ of $G_A$.}
    \label{fig:start-options}
\end{figure}
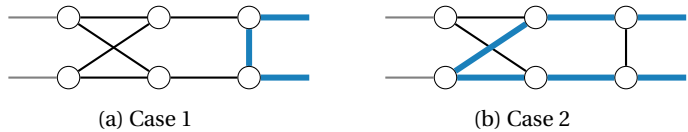
There are only three distinct ways in which $C$ can contain at least one vertex in copy $j$ of $G_A$ (in which case it contains either three, six, or six vertices in copy $j$) as shown in Fig.~\ref{fig:end-options}.

\begin{figure}
    \centering
     \begin{subfigure}[b]{0.3\textwidth}
    \begin{tikzpicture}[main_node/.style={circle,draw,minimum size=1em,inner sep=3pt},scale=0.8]

\node[main_node] (4) at (-1.5, 0) {};
\node[main_node] (1) at (-1.5, 1) {};
\node[main_node] (5) at (0, 0) {};
\node[main_node] (2) at (0, 1) {};
\node[main_node] (6) at (1.5,0) {};
\node[main_node] (3) at (1.5, 1) {};

 \path[draw, thick]
(4) edge node {} (5) 
(5) edge node {} (6) 
(6) edge node {} (3) 
(3) edge node {} (2) 
(2) edge node {} (1) 
(4) edge node {} (2) 
(1) edge node {} (5) 
;

 \path[draw, line width=0.8mm, color=bluereadable]
(1) edge node {} (2) 
(2) edge node {} (4) 

;

\draw[bluereadable, line width=0.8mm] (4) -- (-2.5,0);
\draw[bluereadable, line width=0.8mm] (1) -- (-2.5,1);
\draw[gray, thick] (6) -- (2.5,0);
\draw[gray, thick] (3) -- (2.5,1);

\end{tikzpicture}
    \caption{Case 1}
    \label{fig:case2-4-1}
    \end{subfigure}
     \hfill
     \begin{subfigure}[b]{0.3\textwidth}
    \begin{tikzpicture}[main_node/.style={circle,draw,minimum size=1em,inner sep=3pt},scale=0.8]

\node[main_node] (4) at (-1.5, 0) {};
\node[main_node] (1) at (-1.5, 1) {};
\node[main_node] (5) at (0, 0) {};
\node[main_node] (2) at (0, 1) {};
\node[main_node] (6) at (1.5,0) {};
\node[main_node] (3) at (1.5, 1) {};

 \path[draw, thick]
(4) edge node {} (5) 
(5) edge node {} (6) 
(6) edge node {} (3) 
(3) edge node {} (2) 
(2) edge node {} (1) 
(4) edge node {} (2) 
(1) edge node {} (5) 
;
 \path[draw, line width=0.8mm, color=bluereadable]
(1) edge node {} (5) 
(2) edge node {} (4) 
(2) edge node {} (3)
(5) edge node {} (6)
(3) edge node {} (6) 
;

\draw[bluereadable, line width=0.8mm] (4) -- (-2.5,0);
\draw[bluereadable, line width=0.8mm] (1) -- (-2.5,1);
\draw[gray, thick] (6) -- (2.5,0);
\draw[gray, thick] (3) -- (2.5,1);

\end{tikzpicture}
    \caption{Case 2}
    \label{fig:case2-4-2}
    \end{subfigure}
    \hfill
    \begin{subfigure}[b]{0.3\textwidth}
    \begin{tikzpicture}[main_node/.style={circle,draw,minimum size=1em,inner sep=3pt},scale=0.8]

\node[main_node] (4) at (-1.5, 0) {};
\node[main_node] (1) at (-1.5, 1) {};
\node[main_node] (5) at (0, 0) {};
\node[main_node] (2) at (0, 1) {};
\node[main_node] (6) at (1.5,0) {};
\node[main_node] (3) at (1.5, 1) {};

 \path[draw, thick]
(4) edge node {} (5) 
(5) edge node {} (6) 
(6) edge node {} (3) 
(3) edge node {} (2) 
(2) edge node {} (1) 
(4) edge node {} (2) 
(1) edge node {} (5) 
;

 \path[draw, line width=0.8mm, color=bluereadable]
(1) edge node {} (2) 
(2) edge node {} (3) 
(4) edge node {} (5)
(5) edge node {} (6)
(3) edge node {} (6) 
;

\draw[bluereadable, line width=0.8mm] (4) -- (-2.5,0);
\draw[bluereadable, line width=0.8mm] (1) -- (-2.5,1);
\draw[gray, thick] (6) -- (2.5,0);
\draw[gray, thick] (3) -- (2.5,1);
\end{tikzpicture}
    \caption{Case 3}
    \label{fig:case2-4-3}
    \end{subfigure}
    \hfill
    \vspace{-2mm}
    \caption{The three distinct cases in which $C$ contains some vertices in copy $j$ of $G_A$.}
    \label{fig:end-options}
    \vspace{-7.5mm}
\end{figure}
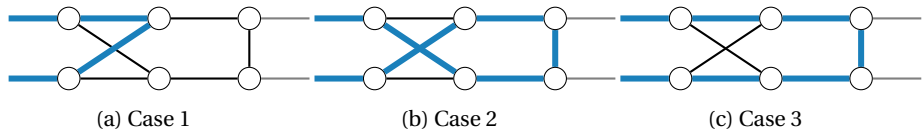

As a result, the length of $C$ is congruent to $2$ modulo $3$.
\end{proof}

We now complement the previous lemma by showing which cycle lengths are present in $F_k$.
\begin{lemma}
\label{lem:FCycPresent}
For each integer $k \geq 2$, $F_k$ contains a cycle of length $\ell$ for each $\ell \in \{4\} \cup \{3m+2: m \in [k-1]\} \cup [3k,6k]$.
\end{lemma}
\begin{proof}
The cycle $v_{1,1}v_{1,2}v_{1,4}v_{1,5}v_{1,1}$ is a cycle of length four in $F_k$. For each $m \in [k-1]$ one can obtain a cycle of length $3m+2$ by combining (i) one of the subpaths shown in Fig.~\ref{fig:start-options} (containing $2$ or $5$ vertices) in copy $1$ of the graph $G_A$, (ii) one of the subpaths shown in Fig.~\ref{fig:cross-options} (containing $6$ vertices) in copies $2, 3, \ldots, \lfloor\frac{m+1}{2}\rfloor$ of the graph $G_A$ and (iii) the subpath shown in Fig.~\ref{fig:case2-4-1} (containing three vertices) in copy $\lfloor\frac{m+1}{2}\rfloor+1$ of the graph $G_A$. 

Finally, we show that there exists a cycle of length $m$ in $F_k$ for each $m \in [3k,6k]$. Let $C$ be the cycle $v_{1,1}v_{1,2}v_{1,3}v_{2,1}v_{2,2}\ldots v_{k,1}v_{k,2}v_{k,3}v_{1,1}$ of length $3k$ in $F_k$. For any $i \in [k]$, replacing the subpath $v_{i,1}v_{i,2}v_{i,3}$ with (i) $v_{i,1}v_{i,5}v_{i,6}v_{i,3}$, (ii) $v_{i,1}v_{i,5}v_{i,4}v_{i,2}v_{i,3}$ or (iii) $v_{i,1}v_{i,2}v_{i,4}v_{i,5}v_{i,6}v_{i,3}$ increases the length of $C$ by $1$, $2$ or $3$, respectively. By consecutively applying these operations for each $i \in [k]$, one can construct a cycle of the desired length.
\end{proof}

We combine Lemmas~\ref{lem:FHamConn},~\ref{lem:FCycNotPresent}, and~\ref{lem:FCycPresent} and obtain the main result of this section:

\begin{theorem}
\label{thm:HamConnGaps}
    For each integer $k \geq 2$ there exists a cubic hamiltonian-connected graph on $6k$ vertices with cycle spectrum $\{4\} \cup \{3m+2: m \in [k-1]\} \cup [3k,6k]$.
\end{theorem}

\section{Concluding remarks}
\label{sec:conclusion}
There are several natural ways to adapt the Faudree-Schelp Conjecture. Our first question is motivated by Theorem~\ref{thm:HFamily} and Thomassen's counterexamples~\cite{thomassen1978counterexamples}.

\begin{ques}
Do there exist real numbers $r_1$ and $r_2$ with $\frac{2}{3} \leq r_1 \leq r_2 < 1$ such that every hamiltonian-connected graph on $n$ vertices satisfies $\mathfrak{P}_k$ for all $k$ such that $r_1n \leq k \leq r_2n$ (whenever $n$ is large enough)?
\end{ques}


We recall that, for $k \ge 3$, $\mathfrak{P}_k$ implies the presence of a $k$-cycle. In the light of Theorem~\ref{thm:HamConnGaps} we wonder whether hamiltonian-connected graphs may have gaps in the second half of their cycle spectrum, i.e.\ the existence of counterexamples to the Faudree-Schelp Conjecture in the sense that large cycles are entirely missing from the spectrum. In other words:
\begin{ques}
Does there exist an $n$-vertex hamiltonian-connected graph whose cycle spectrum does not contain an integer $\ell$ with $\frac{n}{2} \leq \ell \leq n-1$? More generally, among all hamiltonian-connected graphs, how small can the largest (relative to the graph's order) contiguous subset of the cycle spectrum be? 
\end{ques}

Regarding the last question, we mention a related result. For planar 4-connected graphs without 4-cycles (which are hamiltonian-connected~\cite{T83}), Lo proved that the second half of their cycle spectrum is contiguous~\cite{L24}. 
It might also be interesting to determine if there exist hamiltonian-connected graphs on $n$ vertices for which the cardinality of their cycle spectrum is smaller than $2n/3$.

\section*{Acknowledgements}
The authors would like to thank Jarne Renders for early prototype validation and assistance with his software. The computational resources and services used in this work were provided by the VSC (Flemish Supercomputer Center), funded by the Research Foundation Flanders (FWO) and the Flemish Government - department EWI. The research of Jan Goedgebeur and Michiel Provoost was supported by Internal Funds of KU Leuven and an FWO grant with grant number G0AGX24N. Jorik Jooken is supported by a Postdoctoral Fellowship of the Research Foundation Flanders (FWO) with grant number 1222524N. 
\bibliographystyle{splncs04}
\bibliography{mybib.bib}

\begin{thebibliography}{10}
\providecommand{\url}[1]{\texttt{#1}}
\providecommand{\urlprefix}{URL }
\providecommand{\doi}[1]{https://doi.org/#1}

\bibitem{almohanna2018hamiltonian}
Almohanna, N., Olejniczak, D., Zhang, P.: Hamiltonian-connected graphs with
  additional properties. Congressus Numerantium  \textbf{231},  291--302 (2018)

\bibitem{plantri}
Brinkmann, G., McKay, B.D.: Fast generation of planar graphs. MATCH
  Communications in Mathematical and in Computer Chemistry  \textbf{58},
  323--357 (2007)

\bibitem{survey}
Chartrand, G., Zhang, P.: Uniformly connected graphs—a survey. Parallel
  Processing Letters  \textbf{30}(03),  2040002 (2020)

\bibitem{FRS73}
Faudree, R., Rosseau, C., Schelp, R.: Path length distributions {I}. Discrete
  Mathematics  \textbf{6},  35--52 (1973)

\bibitem{FS74}
Faudree, R., Schelp, R.: Path connected graphs. Acta Mathematica Academiae
  Scientiarum Hungarica  \textbf{25},  313--319 (1974)

\bibitem{FS}
Faudree, R., Schelp, R.: The square of a block is strongly path connected.
  Journal of Combinatorial Theory, Series B  \textbf{20}(1),  47--61 (1976)

\bibitem{F76}
Fleischner, H.: In the square of graphs, hamiltonicity and pancyclicity,
  hamiltonian connectedness and panconnectedness are equivalent concepts.
  Monatshefte f\"ur Mathematik  \textbf{82},  125--149 (1976)

\bibitem{github}
Goedgebeur, J., Jooken, J., Provoost, M., Zamfirescu, C.T.: {Counterexamles
  Faudree Schelp} (version 1) (2025), available at
  \url{https://github.com/AGT-Kulak/Counterexamples-Faudree-Schelp}

\bibitem{Jorik}
Jooken, J.: Improved asymptotic upper bounds for the minimum number of longest
  cycles in regular graphs. Discrete Applied Mathematics  \textbf{356},
  133--141 (2024)

\bibitem{K79}
Kotzig, A.: Selected open problems in graph theory. Graph Theory and Related
  Topics (Proc. Conf., Waterloo, 1977, eds.: J.A. Bondy and U.S.R. Murty),
  xxxii+371. Academic Press, New York pp. 358--367 (1979)

\bibitem{L24}
Lo, O.H.S.: Cycles in 3-connected claw-free planar graphs and 4-connected
  planar graphs without 4-cycles. Journal of Graph Theory  \textbf{107}(4),
  702--728 (2024)

\bibitem{nauty}
McKay, B.D., Piperno, A.: Practical graph isomorphism, {II}. Journal of
  Symbolic Computation  \textbf{60},  94--112 (2014)

\bibitem{thomassen1978counterexamples}
Thomassen, C.: Counterexamples to {F}audree and {S}chelp's conjecture on
  hamiltonian-connected graphs. Journal of Graph Theory  \textbf{2}(4),
  341--347 (1978)

\bibitem{T83}
Thomassen, C.: A theorem on paths in planar graphs. Journal of Graph Theory
  \textbf{7}(2),  169--176 (1983)

\bibitem{W16}
Winter, M.: Notes on ${P}(k)$-graphs and the conjecture of {K}otzig. MSc
  Thesis, Chemnitz University of Technology  (2016)

\end{thebibliography}

\newpage
\section*{Appendix}

We found the graphs that led to the ideas for the constructions 
 as well as the fact that the graph in Fig.~\ref{fig:smallest} is the smallest counterexample (both in terms of order and size) by exhaustively generating graphs. The generation itself was done using \verb|geng|~\cite{nauty} for all simple graphs up to order 11 and using \verb|plantri|~\cite{plantri} for several subclasses of planar graphs up to higher orders as specified in the tables below. A two-stage filter was developed to find which graphs were counterexamples. In the next section we expand on this filter, which is available at \cite{github}.

\subsection*{Filtering graphs}
The filter used operates in two stages to allow independent verification of each stage. 

In the first stage graphs are filtered that are hamiltonian-connected. We employ a backtracking algorithm that recursively extends paths with a given start vertex. It maintains two data structures: one stores if a hamiltonian path has been found between the vertices (using a matrix $M$, where $M_{i,j}=1$ if a hamiltonian $ij$-path has been found and $M_{i,j}=0$ otherwise), whereas the other stores the vertices that are in the current path. Additionally, the program also stores what the current last vertex in the path is. The hamiltonian paths are constructed using a depth-first search: the program starts by selecting a vertex $v_{\text{start}}$ for which there is still a distinct vertex $v$ so that $M_{v_{\text{start}},v}=0$. Then all possible paths from $v_{\text{start}}$ are constructed by recursively adding one vertex to the current path and $M_{v_{\text{start}},v'}$ is set to $1$ whenever a hamiltonian $v_{\text{start}}v'$-path has been found (note that $v'$ can be distinct from $v$). Let $H$ be the graph induced by all vertices that are not in the current path. Pruning rules are used to speed up the search. More precisely, the algorithm backtracks in the following cases:
\begin{itemize}
    \item If $H$ is disconnected.
    \item If $H$ has at least three vertices of degree one.
    \item If $H$ has two vertices of degree one, both of which are not adjacent (in the original graph) with the last vertex in the current path.
    \item If $H$ has one vertex $v$ of degree one which is not adjacent (in the original graph) with the last vertex in the current path and $M_{v_{\text{start}},v}=1$.
    \item If $M_{v_{\text{start}},v}=1$ for all $v \in V(H)$.
\end{itemize}
In the second stage, graphs are filtered that do not satisfy ${\mathfrak P}_k$ for some $k$. To determine if a graph satisfies ${\mathfrak P}_k$, the algorithm stores a set of lengths for each pair of vertices in the graph, thereby indicating whether or not a path of the specified length exists between the two corresponding vertices. In this stage we exhaustively verify that ${\mathfrak P}_k$ is satisfied with a depth-first search similar to the one in stage one (but this time the algorithm will backtrack when the length of the current path is at least $k$). Important here is that we start with $k=n-1$ (where $n$ is the order of the graph) and iteratively check for shorter lengths. This way the subpaths of shorter lengths are already stored while searching for paths of longer lengths if they are found. 

Stage two allows for different restrictions on when a graph passes the filter: If it uses the \texttt{all} argument, every graph passes stage two. If it uses the \texttt{full} argument, every graph that satisfies ${\mathfrak P}_k$ for every $3\leq k \leq n$ passes stage two. If it uses the \texttt{any} argument, every graph that does not satisfy ${\mathfrak P}_k$ for some $3\leq k \leq n$ passes stage two. Lastly, if it uses the \texttt{last} argument, every graph that does not satisfy ${\mathfrak P}_k$ for some $n/2+1\leq k \leq n$ passes stage two.

\subsection*{Overview of generated graphs and counterexamples}

The graph generator \verb|plantri|~\cite{plantri} was used for generating all planar 3-connected graphs up to order 13, all planar 3-connected cubic graphs up to order 26 and several other planar graph classes as specified in Table~\ref{table:planar_3_connected_cubic} and Table~\ref{table:no_counterexamples}. Thomassen~\cite{T83} proved that all planar 4-connected graphs are hamiltonian-connected and the classes mentioned in Table~\ref{table:no_counterexamples} are all related to this class. The counts were verified with an independent (but slower) implementation adapted from~\cite{Jorik} up to order 11 for the planar 3-connected case and up to order 20 for the planar 3-connected cubic case. 

In Table~\ref{table:2_connected} we used \verb|geng|~\cite{nauty} for generating all 2-connected graphs up to order 11. The counts in that table were verified with the independent implementation up to order 9. The fastest algorithm was used on a system with 719 threads for up to one CPU-year to check for all graph classes. The bottleneck of the algorithm is the second phase where we verify that the graphs do not satisfy ${\mathfrak P}_k$ for some $n/2+1 \leq k \leq n$. For the independent verification, we adapted the algorithm from \cite{Jorik} that generates all cycles of a graph. Both algorithms produced the same results. The code is made publicly available on~\cite{github}.

We note that, with the exception of the graphs in Table~\ref{table:2_connected}, all graphs that we checked are 3-connected, since hamiltonian-connected graphs on at least 4 vertices are necessarily 3-connected. The reason that we did check 2-connected graphs in the non-planar case is because \texttt{geng} has no support to only generate 3-connected graphs. 

\begin{table}[h!]
\centering
\begin{tabular}{|c|c|c|c|c|}
\hline
\textbf{Graph class} & \textbf{~Order~}& \textbf{~\#Graphs~} & \textbf{~\#Ham-conn. graphs~} & \textbf{~\#Counterexamples~} \\
\hline
\multirow{6}{*}{~Planar 3-connected~} & 8 & 257 & 246 & 2 \\
 & 9 & 2 606 & 2 526 & 0 \\
 & 10 & 32 300 & 30 842 & 31 \\
 & 11 & 440 564 & 416 108 & 8 \\
 & 12 & 6 384 634 & 5 955 716 & 271 \\
 & 13 & 96 262 938 & 88 766 610 & 214 \\
\hline
\multirow{12}{*}{Planar 3-connected cubic} & 8 & 2 & 1 & 1 \\
 & 10 & 5 & 4 & 1 \\
 & 12 & 14 & 9 & 4 \\
 & 14 & 50 & 35 & 6 \\
 & 16 & 233 & 151 & 14 \\
 & 18 & 1 249 & 826 & 34 \\
 & 20 & 7 595 & 4 833 & 55 \\
 & 22 & 49 566 & 30 904 & 225 \\
 & 24 & 339 722 & 206 022 & 308 \\
 & 26 & 2 406 841 & 1 427 986 & 1 822  \\
\hline
\end{tabular}
\caption{Planar 3-connected and planar 3-connected cubic graphs. These were generated with \texttt{plantri}~\cite{plantri} and filtered using our program~\cite{github}.}
\label{table:planar_3_connected_cubic}
\end{table}

\begin{table}[h!]
\centering
\begin{tabular}{|c|c|c|c|c|}
\hline
\textbf{Graph class} & \textbf{~Order~} & \textbf{~\#Graphs~} & \textbf{~\#Ham-conn. graphs~} & \textbf{~\#Counterexamples~} \\
\hline
Planar 3-connected 4-regular & $\leq 18$ & 6 473 & 6 473 & 0 \\
\hline
Planar 3-connected minimal degree 4~ & $\leq 13$ & 35 793 & 35 793 & 0 \\
\hline
Planar 3-connected minimal degree 5~ & $\leq 18$ & 46 & 46 & 0 \\
\hline
Planar 4-connected & $\leq 15$ & 1 759 602 & 1 759 602 & 0 \\
\hline
Planar 3-connected with girth at least 4 & $\leq 15$ & 394 & 71 & 0 \\
\hline
\end{tabular}
\caption{Planar graph classes without counterexamples up to the given orders. These were generated with \texttt{plantri}~\cite{plantri} and some were filtered using \texttt{pickg}~\cite{nauty}. All were then filtered using our program~\cite{github}.}
\label{table:no_counterexamples}
\end{table}

\begin{table}[h!]
\centering
\begin{tabular}{|c|c|c|c|c|}
\hline
\textbf{Graph class} & \textbf{~Order~}& \textbf{~\#Graphs~} & \textbf{~\#Ham-conn. graphs~} & \textbf{~\#Counterexamples~} \\
\hline
\multirow{8}{*}{~2-connected graphs~} & 3 & 1 & 1 & 0 \\
& 4 & 3 & 1 & 0 \\
& 5 & 10 & 3 & 0 \\
& 6 & 56 & 13 & 0 \\
& 7 & 468 & 116 & 0 \\
& 8 & 7 123 & 2 009 & 3 \\
& 9 & 194 066 & 72 529 & 0 \\
& 10 & 9 743 542 & 4 784 340 & 141 \\
& 11 & 900 969 091 & 397 488 917 & 15 \\
\hline
\end{tabular}
\caption{2-connected graphs. These were generated with \texttt{geng}~\cite{nauty} and filtered using our program~\cite{github}. }
\label{table:2_connected}
\end{table}

\clearpage
\end{document}